\documentclass[a4paper]{article}
\usepackage[active]{srcltx}
\usepackage{indentfirst,latexsym,bm,amsmath,pstricks,amssymb,amsthm,amscd}
\usepackage[all]{xy}
\begin{document}
\theoremstyle{plain}
\newtheorem{thm}{Theorem}[section]
\newtheorem{theorem}[thm]{Theorem}
\newtheorem{addendum}[thm]{Addendum}
\newtheorem{lemma}[thm]{Lemma}
\newtheorem{corollary}[thm]{Corollary}
\newtheorem{proposition}[thm]{Proposition}
\theoremstyle{definition}
\newtheorem{remark}[thm]{Remark}
\newtheorem{remarks}[thm]{Remarks}
\newtheorem{notations}[thm]{Notations}
\newtheorem{definition}[thm]{Definition}
\newtheorem{claim}[thm]{Claim}
\newtheorem{assumption}[thm]{Assumption}
\newtheorem{assumptions}[thm]{Assumptions}
\newtheorem{property}[thm]{Property}
\newtheorem{properties}[thm]{Properties}
\newtheorem{example}[thm]{Example}
\newtheorem{examples}[thm]{Examples}
\newtheorem{conjecture}[thm]{Conjecture}
\newtheorem{questions}[thm]{Questions}
\newtheorem{question}[thm]{Question}
\numberwithin{equation}{section}
\newcommand{\sA}{{\mathcal A}}
\newcommand{\sB}{{\mathcal B}}
\newcommand{\sC}{{\mathcal C}}
\newcommand{\sD}{{\mathcal D}}
\newcommand{\sE}{{\mathcal E}}
\newcommand{\sF}{{\mathcal F}}
\newcommand{\sG}{{\mathcal G}}
\newcommand{\sH}{{\mathcal H}}
\newcommand{\sI}{{\mathcal I}}
\newcommand{\sJ}{{\mathcal J}}
\newcommand{\sK}{{\mathcal K}}
\newcommand{\sL}{{\mathcal L}}
\newcommand{\sM}{{\mathcal M}}
\newcommand{\sN}{{\mathcal N}}
\newcommand{\sO}{{\mathcal O}}
\newcommand{\sP}{{\mathcal P}}
\newcommand{\sQ}{{\mathcal Q}}
\newcommand{\sR}{{\mathcal R}}
\newcommand{\sS}{{\mathcal S}}
\newcommand{\sT}{{\mathcal T}}
\newcommand{\sU}{{\mathcal U}}
\newcommand{\sV}{{\mathcal V}}
\newcommand{\sW}{{\mathcal W}}
\newcommand{\sX}{{\mathcal X}}
\newcommand{\sY}{{\mathcal Y}}
\newcommand{\sZ}{{\mathcal Z}}
\newcommand{\A}{{\mathbb A}}
\newcommand{\B}{{\mathbb B}}
\newcommand{\C}{{\mathbb C}}
\newcommand{\D}{{\mathbb D}}
\newcommand{\E}{{\mathbb E}}
\newcommand{\F}{{\mathbb F}}
\newcommand{\G}{{\mathbb G}}
\newcommand{\BH}{{\mathbb H}}
\newcommand{\I}{{\mathbb I}}
\newcommand{\J}{{\mathbb J}}
\newcommand{\BL}{{\mathbb L}}
\newcommand{\M}{{\mathbb M}}
\newcommand{\N}{{\mathbb N}}
\newcommand{\BP}{{\mathbb P}}
\newcommand{\Q}{{\mathbb Q}}
\newcommand{\R}{{\mathbb R}}
\newcommand{\BS}{{\mathbb S}}
\newcommand{\T}{{\mathbb T}}
\newcommand{\U}{{\mathbb U}}
\newcommand{\V}{{\mathbb V}}
\newcommand{\W}{{\mathbb W}}
\newcommand{\X}{{\mathbb X}}
\newcommand{\Y}{{\mathbb Y}}
\newcommand{\Z}{{\mathbb Z}}
\newcommand{\rk}{{\rm rk}}
\newcommand{\ch}{{\rm c}}
\newcommand{\Sp}{{\rm Sp}}
\newcommand{\Sl}{{\rm Sl}}
\title{{Inequalities between the Chern numbers of a singular
fiber in a family of algebraic curves} 
\thanks{This work is supported by NSFC,
Foundation of the EMC and the Program of Shanghai Subject Chief
Scientist.} }

\author{Jun Lu, Sheng-Li Tan}

\noindent
\maketitle
\begin{abstract}
{In a family of curves, the Chern numbers of a singular fiber are the local contributions to the
 Chern numbers of the total space. We will give some inequalities
between the Chern numbers of a singular fiber as well as their lower and upper bounds.
We introduce the dual fiber of a singular fiber, and prove a duality theorem. As an
application, we will classify singular fibers with large or small Chern numbers.}
\end{abstract}

\section{Introduction and main results}
Chern numbers of a singular fiber in a family of curves are the
local contributions of the fiber to the global Chern numbers of the
total space. Our first purpose of this paper is to find the best
inequalities between the Chern numbers of a singular fiber. Our
second purpose is to try to give a new approach to the
classification of singular fibers of genus $g$.  We know that when
 $g$ is big, there are too many singular fibers of genus
$g$ to classify completely (see \cite{Iitaka}, \cite{NaUe},
\cite{Ogg}, \cite{Uematsu}). In order to get the local-global
relations between the invariants, one possible way is to classify
singular fibers according to their contributions to the global
invariants. To explain this approach, we will classify singular
fibers with big or small Chern numbers and give some applications.
See the survey \cite{AK00} for the background of the study on the
local-global properties for families of curves.

A family of curves of genus $g$ over $C$ is a fibration $f:X\to C$
whose general fibers $F$ are smooth curves of genus $g$, where $X$
is a complex smooth projective surface. The family is called {\it
semistable} if all of the singular fibers are reduced nodal curves.
If $X=F\times C$ and $f$ is just the second projection to $C$, then
we call $f$ a {\it trivial} family. If all of the smooth fibers of
$f$ are isomorphic to each other, equivalently, $f$ becomes trivial
under a finite base change $\tilde{C}\to C$, then $f$ is called {\it
isotrivial}. We always assume that $f$ is relatively minimal, i.e.,
there is no $(-1)$-curve in any singular fiber.

When $g=1$,  Kodaira \cite{Kod1963III} found the global invariants
from the singular fibers. The first Chern number $c_1^2(X)$ is
always zero, the second Chern number $c_2(X)$ is equal to
$12\chi(\mathcal O_X)$ by Noether's formula, and
\begin{equation}\label{0.1}
c_2(X)=j+6\nu(\textrm{I}^*)+2\nu(\textrm{II})+10\nu(\textrm{II}^*)+
3\nu(\textrm{III})+9\nu(\textrm{III}^*)+
4\nu(\textrm{IV})+8\nu(\textrm{IV}^*),
\end{equation}
where $\nu(\textrm{T})$ denotes the number of singular fibers of
type T, and $j$ is the number of poles of the $J$-function of the
family. Note that the $J$-function over $C$ induces a holomorphic
map of degree $j$ from $C$ to the moduli space $\overline{ \mathcal
M}_1$ of elliptic curves.  So $j$ depends only on the generic
fibers.

By introducing the Chern numbers $c_1^2(F)$, $c_2(F)$ and $\chi_F$
for a singular fiber $F$, the second author (\cite{Ta96},
\cite{Ta98}, \cite{Ta10}) generalized Kodaira's formula (\ref{0.1})
to the higher genus case,
\begin{equation}\label{0.3}
\begin{cases}
c_1^2(X)=\kappa(f)+8(g-1)(g(C)-1)+\sum_{i=1}^sc_1^2(F_i), & \\
c_2(X)=\delta(f)+4(g-1)(g(C)-1)+\sum_{i=1}^sc_2(F_i),&\\
\chi(\mathcal O_X)=\lambda(f)+ (g-1)(g(C)-1)+\sum_{i=1}^s\chi_{F_i},
&
\end{cases}
\end{equation}
where $F_1,\cdots,F_s$ are all singular fibers of $f$, and
$\kappa(f)$, $\delta(f)$ and $\lambda(f)$ are the {\it modular
invariants} of the family.  $f$ induces also a holomorphic map from
$C$ to the moduli space of semistable curves of genus $g$:
$$
J: C\longrightarrow \overline{\mathcal M}_g\, .
$$
Then $\kappa(f)=\deg J^*\kappa$, $\delta(f)=\deg J^*\delta$ and
$\lambda(f)=\deg J^*\lambda$, where $\lambda$, $\delta$ and $\kappa$
are the Hodge divisor class, the boundary divisor class and $\kappa
= 12\lambda -\delta$. In the case of elliptic fibrations,
$\kappa(f)=0$ and $\delta(f)=j$.

Let $\widetilde f: \widetilde X \to \widetilde C$ be a semistable
reduction of $F$ under any base change
 $\pi:\widetilde C\to
C$  ramified over $p=f(F)$ and some non-critical points of $f$. The
Chern numbers of $F$ are defined as follows.
\begin{equation}\label{1.1}
c_1^2(F)=K_{f}^2-\dfrac1d K_{\widetilde f}^2, \hskip0.3cm
c_2(F)=e_f-\dfrac1d e_{\widetilde f}, \hskip0.3cm
\chi_F=\chi_f-\dfrac1d \chi_{\widetilde f},
\end{equation}
where  $d$ is the degree of $\pi$, and $
K_f^2=c_1^2(X)-8(g-1)(g(C)-1)$,
              $e_f=c_2(X)-4(g-1)(g(C)-1)$ and
              $\chi_f=\chi(\mathcal O_X)-(g-1)(g(C)-1)$ are the
              relative invariants of $f$.
These Chern numbers are independent of the choice of the semistable
reduction $\pi$. If $g=1$, then $c_1^2(F)=0$ and $c_2(F)$ is exactly
the coefficient in (\ref{0.1}) according to the type of the fiber
$F$. See \S~\ref{S3.3} for the computation formulas for  the Chern
numbers of $F$. We summarize briefly the known properties of the
Chern numbers. Assume that $g=g(F)\geq 2$ and $F$ contains no
$(-1)$-curves.

\begin{enumerate}
\item {\it Positivity}: \ $c_1^2(F)$, $c_2(F)$ and $\chi_F$ are non-negative rational
numbers, one of the three numbers vanishes if and only if $F$ is
semistable.

\item {\it Noether's equality}: \  $ c_1^2(F)+c_2(F)=12\chi_F.$

\item {\it Blow-up formulas}: \  $
c_1^2(F')=c_1^2(F)-1, \hskip0.1cm c_2(F')=c_2(F)+1, \hskip0.1cm
\chi_{F'}=\chi_F, $ where $F'=\sigma^*F$ is the pullback of $F$
under a blowing up $\sigma: X'\to X$  at a point $p$ on $F$.

\item {\it Canonical class inequality}: $c_1^2(F)\leq 4g-4$.

\item {\it Miyaoka-Yau type inequality}: \ $
c_1^2(F)\leq 2c_2(F)$, {or equivalently} $c_1^2(F)\leq 8\chi_F$,
 with equality iff $F_{\rm
 red}$ is a nodal curve and $F=nF_{\rm
 red}$ for some positive integer $n$.
\end{enumerate}

The positivity is essentially due to Beauville \cite{Be}, Xiao
\cite{Xi92} and  \cite{Ta94}. Noether's equality and the blow-up
formulas are direct consequences of the definition of Chern numbers.
The last two inequalities can be found in \cite{Ta96}.

Let $\bar F$ be a normal crossing fiber obtained by blowing up the
singularities of $F$. ($\bar F$ is called the normal crossing model
of $F$). Write $\bar F=n_1C_1+\cdots+n_kC_k$, where $C_i$'s are the
irreducible components. Denote by $M_F$ the least common
multiplicity of $n_1$, $\cdots,$ $n_k$. Let $n$ be a positive
integer satisfying $n\equiv -1\pmod{M_F}$. Denote by $F^*$ the fiber
obtained from $F$ by a local base change $\pi$ defined by $w=z^n$.
We call $F^*$ the {\it dual fiber} of $F$ (see \S2). This is a
natural generalization of Kodaira's dual fibers for elliptic
fibrations. Our first result is the duality theorem for $\chi_F$.


\begin{theorem}\label{THMnew1.1} {\bf (Duality theorem for $\chi$)}  \ Let $\bar F$  and
$\bar F^*$ be the normal crossing models of $F$ and $F^*$
respectively. Let $N_{\bar F}=g-p_a(\bar F_{\rm red})$. Then $0\leq
N_{\bar F}\leq g$.
\begin{enumerate}
\item[$1)$] $N_{\bar F}=N_{\bar F^*}$,  i.e.,  $p_a(\bar F_{\rm red})=p_a(\bar F^*_{\rm red})$.
\item[$2)$]  $\chi_{F}+\chi_{F^{*}}=N_{\bar F}$.
\item[$3)$]  $
\frac16{N_{{\bar{F}}}}\leq \chi_F\leq\frac56{N_{{\bar{F}}}}. $  \
$\chi_F=\frac16{N_{{\bar{F}}}}$ (resp. $\frac56{N_{{\bar{F}}}}$) if
and only if $F$ (resp. $F^*$) is a reduced curve whose singularities
 are at worst ordinary cusps or
nodes.
 \end{enumerate}
\end{theorem}

In general, $F^{**}$ is not necessarily equal to $F$, but we have
the equality $\chi_{F^{**}}=\chi_F$.

\begin{theorem}\label{THMnew1.2} Assume that $g\geq 2$.  We have the following optimal inequalities.
\begin{enumerate}
\item[$1)$] If $F$ is not semistable, then $c_2(F)\ge
\frac{11}{6}$ and $\chi_F\ge \frac{1}{6}$. One of the equalities
holds if and only if $F$ is a reduced curve with one ordinary cusp
and some nodes.

\item[$2)$] $c_1^2(F)\leq 4g-\frac{24}5$. More precisely, if $g\geq 7$ or $g=5$, then $c_1^2(F)\leq
4g-\frac{11}2.$
$$
c_1^2(F)\leq\footnotesize\begin{cases}
\frac{16}{5}, & g=2\\
7, & g=3\\
\frac{54}{5}, & g=4\\
\frac{130}{7}, & g=6
\end{cases}
$$

\item[$3)$] {\rm (Arakelov type inequality)} \
$\chi_F\leq \frac{5g}{6}, $ with equality iff $F^*$ is a reduced
curve with nodes and ordinary cusps as its singularities, and its
normal crossing model is a tree of smooth rational curves.
\end{enumerate}
\end{theorem}
For any $g\geq 2$, there is a numerical fiber $F$ with
$c_1^2(F)=4g-\frac{11}2$ (see Example \ref{LiZi}).
\begin{theorem}\label{THMnew1.3} Let $F$ be a minimal singular fiber of
genus $g\geq 2$ satisfying $c_1^2(F)> 4g-\frac{11}2$. Then $g\leq 6$
and $F$ is one of the following $22$ fibers. $\circ$ is a
$(-2)$-curve, and $\bullet$ is a $(-3)$-curve.
\end{theorem}
 \setlength{\unitlength}{0.7mm}

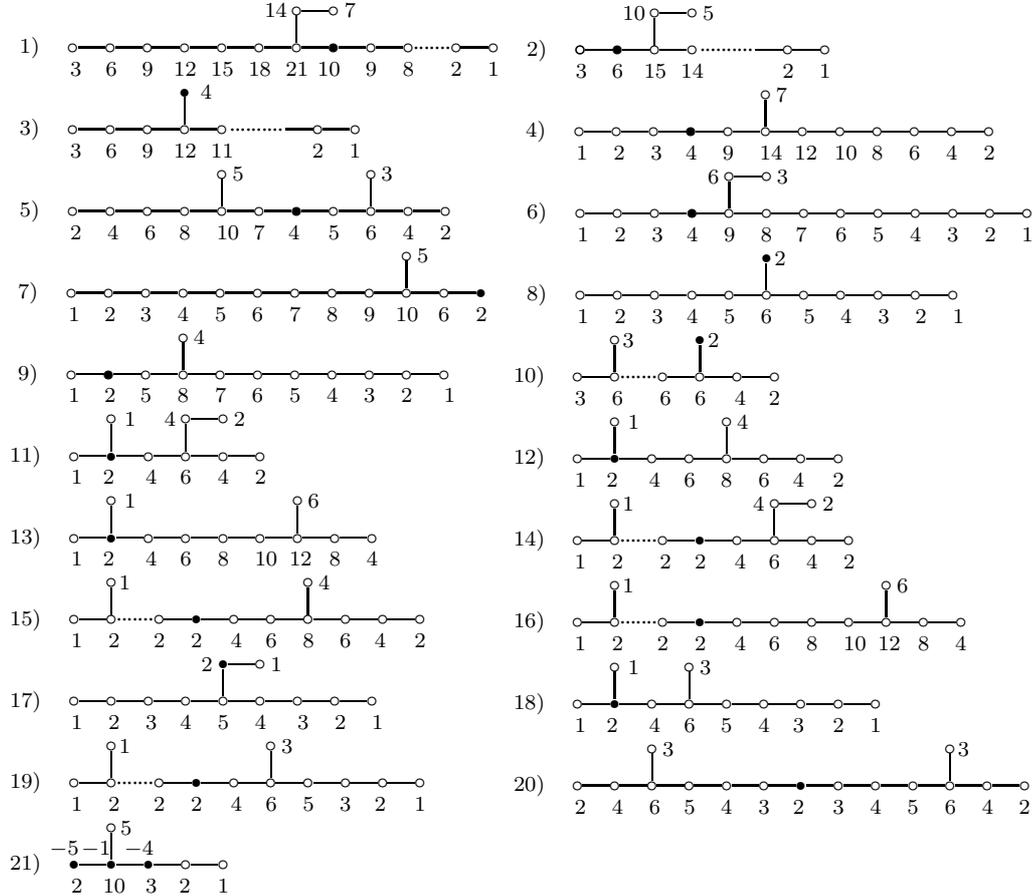
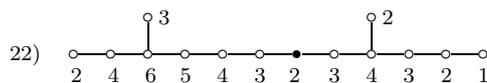
\begin{picture}(200,10)(-4.8,-1)\footnotesize
\put(-35,0.5){\makebox(0,0)[l]{}}
 \multiput(0,0)(7,0){10}{\circle{1.5}}\multiput(1,0)(7,0){9}{\line(1,0){5}}
 \multiput(64.5,0)(1,0){7}{\circle*{0.5}}
\multiput(72,0)(7,0){2}{\circle{1.5}}\put(73,0){\line(1,0){5}}
\multiput(42,7)(7,0){2}{\circle{1.5}}
\multiput(42,1.05)(0,7){1}{\line(0,1){5}}
\multiput(43,7)(0,7){1}{\line(1,0){5}}
 \put(49,0){\circle*{1.5}}
\put(-0.5,-4){\makebox(0,0)[l]{$3$}}
\put(6.2,-4){\makebox(0,0)[l]{$6$}}
\put(13,-4){\makebox(0,0)[l]{$9$}}
\put(19,-4){\makebox(0,0)[l]{$12$}}
\put(26,-4){\makebox(0,0)[l]{$15$}}
\put(33,-4){\makebox(0,0)[l]{$18$}}
\put(40,-4){\makebox(0,0)[l]{$21$}}
\put(36,7){\makebox(0,0)[l]{$14$}}\put(51,7){\makebox(0,0)[l]{$7$}}
\put(46,-4){\makebox(0,0)[l]{$10$}}
\put(55,-4){\makebox(0,0)[l]{$9$}}
\put(62,-4){\makebox(0,0)[l]{$8$}}
\put(71,-4){\makebox(0,0)[l]{$2$}}
\put(78,-4){\makebox(0,0)[l]{$1$}}
\put(-10,0){\makebox(0,0)[l]{$1)$}}
\end{picture}\\[-23pt]

 \setlength{\unitlength}{0.7mm}
\hfill\begin{picture}(90,0)(6.5,-1)\footnotesize
\put(-35,0.5){\makebox(0,0)[l]{}}
\put(7,0){\circle{1.5}}\put(14,0){\circle*{1.5}}\multiput(8,0)(7,0){3}{\line(1,0){5}}
\multiput(7,0)(7,0){4}{\circle{1.5}} \put(21,7){\circle{1.5}}
\multiput(21,1)(0,7){1}{\line(0,1){5}}
\multiput(22,7)(0,7){1}{\line(1,0){5}}
 \put(28,7){\circle{1.5}}
\multiput(30,0)(1,0){10}{\circle*{0.5}}\multiput(40,0)(7,0){2}{\line(1,0){5}}
\multiput(46,0)(7,0){2}{\circle{1.5}}
\put(6.2,-4){\makebox(0,0)[l]{$3$}}
\put(13,-4){\makebox(0,0)[l]{$6$}}
\put(19,-4){\makebox(0,0)[l]{$15$}}
\put(26,-4){\makebox(0,0)[l]{$14$}}
\put(15,7){\makebox(0,0)[l]{$10$}}\put(30,7){\makebox(0,0)[l]{$5$}}
\put(45,-4){\makebox(0,0)[l]{$2$}}
\put(52,-4){\makebox(0,0)[l]{$1$}}
\put(-3,0){\makebox(0,0)[l]{$2)$}}
\end{picture}\\[-23pt]

 \setlength{\unitlength}{0.7mm}
\begin{picture}(200,20)(-4.8,-1)\footnotesize
\put(0,0){\circle{1.5}}
\multiput(1,0)(7,0){4}{\line(1,0){5}}\multiput(7,0)(7,0){4}{\circle{1.5}}
\put(21,7){\circle*{1.5}}\put(21,1){\line(0,1){5}}
\multiput(30,0)(1,0){10}{\circle*{0.5}}\multiput(40,0)(7,0){2}{\line(1,0){5}}\multiput(46,0)(7,0){2}
{\circle{1.5}} \put(-0.5,-4){\makebox(0,0)[l]{$3$}}
\put(6.2,-4){\makebox(0,0)[l]{$6$}}
\put(13,-4){\makebox(0,0)[l]{$9$}}
\put(19,-4){\makebox(0,0)[l]{$12$}}
\put(26,-4){\makebox(0,0)[l]{$11$}}
\put(24,7){\makebox(0,0)[l]{$4$}} \put(45,-4){\makebox(0,0)[l]{$2$}}
\put(52,-4){\makebox(0,0)[l]{$1$}}
\put(-10,0){\makebox(0,0)[l]{$3)$}}
\end{picture}\\[-23pt]

\setlength{\unitlength}{0.7mm}
\hfill\begin{picture}(100,0)(-10.3,-1)\footnotesize
\put(-35,0.5){\makebox(0,0)[l]{}}
\multiput(0,0)(7,0){12}{\circle{1.5}}\multiput(1,0)(7,0){11}{\line(1,0){5}}
\multiput(21,0)(1,0){1}{\circle*{1.5}}
\multiput(35,7)(0,7){1}{\circle{1.5}}
 \multiput(35,1)(0,7){1}{\line(0,1){5}}
\put(-0.5,-4){\makebox(0,0)[l]{$1$}}
\put(6.5,-4){\makebox(0,0)[l]{$2$}}
\put(13.5,-4){\makebox(0,0)[l]{$3$}}
\put(20,-4){\makebox(0,0)[l]{$4$}}
\put(27,-4){\makebox(0,0)[l]{$9$}}
\put(34,-4){\makebox(0,0)[l]{$14$}}
\put(40.5,-4){\makebox(0,0)[l]{$12$}}
\put(48,-4){\makebox(0,0)[l]{$10$}}
 \put(55,-4){\makebox(0,0)[l]{$8$}} \put(62,-4){\makebox(0,0)[l]{$6$}}
 \put(69,-4){\makebox(0,0)[l]{$4$}} \put(76,-4){\makebox(0,0)[l]{$2$}}
\put(37,7){\makebox(0,0)[l]{$7$}}
\put(-10,0){\makebox(0,0)[l]{$4)$}}
\end{picture}\\[-23pt]


\setlength{\unitlength}{0.7mm}
\hfill\begin{picture}(200,20)(-4.8,-1)\footnotesize
\put(-35,0.5){\makebox(0,0)[l]{}}
\put(42,0){\circle*{1.5}}\multiput(0,0)(7,0){7}{\circle{1.5}}
\multiput(49,0)(7,0){4}{\circle{1.5}}\multiput(1,0)(7,0){10}{\line(1,0){5}}
 \multiput(28,1)(0,7){1}{\line(0,1){5}} \multiput(56,1)(0,7){1}{\line(0,1){5}}
 \multiput(28,7)(0,7){1}{\circle{1.5}} \multiput(56,7)(0,7){1}{\circle{1.5}}
\put(-0.5,-4){\makebox(0,0)[l]{$2$}}\put(6.5,-4){\makebox(0,0)[l]{$4$}}
\put(13.5,-4){\makebox(0,0)[l]{$6$}}
\put(20,-4){\makebox(0,0)[l]{$8$}}
\put(27,-4){\makebox(0,0)[l]{$10$}}
\put(34,-4){\makebox(0,0)[l]{$7$}}
\put(40.5,-4){\makebox(0,0)[l]{$4$}}
\put(48,-4){\makebox(0,0)[l]{$5$}}\put(30,7){\makebox(0,0)[l]{$5$}}
\put(55,-4){\makebox(0,0)[l]{$6$}}\put(62,-4){\makebox(0,0)[l]{$4$}}
\put(69,-4){\makebox(0,0)[l]{$2$}}\put(58,7){\makebox(0,0)[l]{$3$}}
\put(-10,0){\makebox(0,0)[l]{$5)$}}
\end{picture}\\[-23pt]

 \setlength{\unitlength}{0.7mm}
 \hfill\begin{picture}(100,0)(-10.5,-1)\footnotesize
 \put(-35,0.5){\makebox(0,0)[l]{}}
 \multiput(0,0)(7,0){13}{\circle{1.5}}
 \multiput(1,0)(7,0){12}{\line(1,0){5}}
 \multiput(21,0)(1,0){1}{\circle*{1.5}}
 \multiput(28,7)(0,7){1}{\circle{1.5}}
  \multiput(35,7)(0,7){1}{\circle{1.5}}
 \multiput(28,1)(0,7){1}{\line(0,1){5}}
 \multiput(29,7)(0,7){1}{\line(1,0){5}}
\put(-0.5,-4){\makebox(0,0)[l]{$1$}}
\put(6.5,-4){\makebox(0,0)[l]{$2$}}
\put(13.5,-4){\makebox(0,0)[l]{$3$}}
\put(20,-4){\makebox(0,0)[l]{$4$}}
\put(27,-4){\makebox(0,0)[l]{$9$}}
\put(34,-4){\makebox(0,0)[l]{$8$}}
\put(40.5,-4){\makebox(0,0)[l]{$7$}}
\put(48,-4){\makebox(0,0)[l]{$6$}}
 \put(55,-4){\makebox(0,0)[l]{$5$}} \put(62,-4){\makebox(0,0)[l]{$4$}}
 \put(69,-4){\makebox(0,0)[l]{$3$}} \put(76,-4){\makebox(0,0)[l]{$2$}}
 \put(83,-4){\makebox(0,0)[l]{$1$}}\put(24,7){\makebox(0,0)[l]{$6$}}
\put(37,7){\makebox(0,0)[l]{$3$}}
\put(-10,0){\makebox(0,0)[l]{$6)$}}
\end{picture}\\[-23pt]

 \setlength{\unitlength}{0.7mm}
\hfill\begin{picture}(200,20)(-4.5,-1)\footnotesize
\put(-35,0.5){\makebox(0,0)[l]{}}
\multiput(7,0)(7,0){10}{\circle{1.5}} \put(63,7){\circle{1.5}}
\put(0,0){\circle{1.5}} \multiput(1,0)(7,0){11}{\line(1,0){5}}
\put(63,1){\line(0,1){5}}\put(77,0){\circle*{1.5}}
\put(-0.5,-4){\makebox(0,0)[l]{$1$}}
\put(6.2,-4){\makebox(0,0)[l]{$2$}}
\put(13,-4){\makebox(0,0)[l]{$3$}}
\put(20,-4){\makebox(0,0)[l]{$4$}}
\put(27,-4){\makebox(0,0)[l]{$5$}}
\put(34,-4){\makebox(0,0)[l]{$6$}}
\put(41,-4){\makebox(0,0)[l]{$7$}}
\put(48,-4){\makebox(0,0)[l]{$8$}}
\put(55,-4){\makebox(0,0)[l]{$9$}}
\put(61,-4){\makebox(0,0)[l]{$10$}}
\put(69,-4){\makebox(0,0)[l]{$6$}}
\put(76,-4){\makebox(0,0)[l]{$2$}}
 \put(65,7){\makebox(0,0)[l]{$5$}}
 \put(-10,0){\makebox(0,0)[l]{$7)$}}
\end{picture}\\[-23pt]


\setlength{\unitlength}{0.7mm}
\hfill\begin{picture}(100,0)(-10.5,-1)\footnotesize
\put(-35,-0.5){\makebox(0,0)[l]{}}
\multiput(0,0)(7,0){11}{\circle{1.5}}
 \multiput(1,0)(7,0){10}{\line(1,0){5}}\multiput(35,7)(0,7){1}{\circle*{1.5}}
 \multiput(35,1)(0,7){1}{\line(0,1){5}}
\put(-0.5,-4){\makebox(0,0)[l]{$1$}}
\put(6.5,-4){\makebox(0,0)[l]{$2$}}\put(13.5,-4){\makebox(0,0)[l]{$3$}}
\put(20,-4){\makebox(0,0)[l]{$4$}}
\put(27,-4){\makebox(0,0)[l]{$5$}}\put(49,-4){\makebox(0,0)[l]{$4$}}
\put(34,-4){\makebox(0,0)[l]{$6$}}\put(42,-4){\makebox(0,0)[l]{$5$}}
 \put(56,-4){\makebox(0,0)[l]{$3$}}\put(63,-4){\makebox(0,0)[l]{$2$}}
\put(70,-4){\makebox(0,0)[l]{$1$}}\put(36.5,7){\makebox(0,0)[l]{$2$}}
\put(-10,0){\makebox(0,0)[l]{$8)$}}
\end{picture}\\[-23pt]

 \setlength{\unitlength}{0.7mm}
\hfill\begin{picture}(200,20)(-4.5,-1)\footnotesize
\put(-35,0.5){\makebox(0,0)[l]{}}
\multiput(0,0)(7,0){11}{\circle{1.5}} \put(21,7){\circle{1.5}}
  \multiput(1,0)(7,0){10}{\line(1,0){5}}
\put(21,1){\line(0,1){5}}\put(7,0){\circle*{1.5}}
\put(-0.5,-4){\makebox(0,0)[l]{$1$}}
\put(6.2,-4){\makebox(0,0)[l]{$2$}}
\put(13,-4){\makebox(0,0)[l]{$5$}}
\put(20,-4){\makebox(0,0)[l]{$8$}}
\put(27,-4){\makebox(0,0)[l]{$7$}}
\put(34,-4){\makebox(0,0)[l]{$6$}}
\put(41,-4){\makebox(0,0)[l]{$5$}}
\put(48,-4){\makebox(0,0)[l]{$4$}}
\put(55,-4){\makebox(0,0)[l]{$3$}}
\put(62,-4){\makebox(0,0)[l]{$2$}}
\put(70,-4){\makebox(0,0)[l]{$1$}}
 \put(23,7){\makebox(0,0)[l]{$4$}}
 \put(-10,0){\makebox(0,0)[l]{$9)$}}
\end{picture}\\[-23pt]

 \setlength{\unitlength}{0.7mm}
\hfill\begin{picture}(100,0)(-26,-1)\footnotesize
\put(-35,0.5){\makebox(0,0)[l]{}}
\multiput(-16,0)(7,0){2}{\circle{1.5}}\multiput(-15,0)(7,0){1}{\line(1,0){5}}
\multiput(-9,7)(0,7){1}{\circle{1.5}}\multiput(-9,1)(0,7){1}{\line(0,1){5}}
\multiput(-7.5,0)(1,0){7}{\circle*{0.5}}\multiput(0,0)(7,0){4}{\circle{1.5}}
 \multiput(1,0)(7,0){3}{\line(1,0){5}}\multiput(7,1)(0,7){1}{\line(0,1){5}}
 \multiput(7,7)(0,7){1}{\circle*{1.5}}
\put(-16.5,-4){\makebox(0,0)[l]{$3$}}\put(-9.5,-4){\makebox(0,0)[l]{$6$}}
\put(-0.5,-4){\makebox(0,0)[l]{$6$}}
\put(6,-4){\makebox(0,0)[l]{$6$}}
\put(13.5,-4){\makebox(0,0)[l]{$4$}}
\put(20,-4){\makebox(0,0)[l]{$2$}}
\put(-7.5,7){\makebox(0,0)[l]{$3$}}
 \put(8.5,7){\makebox(0,0)[l]{$2$}}
 \put(-28,0){\makebox(0,0)[l]{$10)$}}
\end{picture}\\[-23pt]

\setlength{\unitlength}{0.7mm}
\hfill\begin{picture}(200,20)(-5,-1)\footnotesize
\put(-35,-0.5){\makebox(0,0)[l]{}}
\multiput(0,0)(7,0){1}{\circle{1.5}}
\multiput(7,0)(7,0){1}{\circle*{1.5}}
\multiput(7,7)(0,7){1}{\circle{1.5}}
\multiput(7,1)(0,7){1}{\line(0,1){5}}
\multiput(14,0)(7,0){4}{\circle{1.5}}
 \multiput(1,0)(7,0){5}{\line(1,0){5}}
 \multiput(21,7)(0,7){1}{\circle{1.5}}
 \multiput(28,7)(0,7){1}{\circle{1.5}}
 \multiput(22,7)(0,7){1}{\line(1,0){5}}
 \multiput(21,1)(0,7){1}{\line(0,1){5}}
\put(-0.5,-4){\makebox(0,0)[l]{$1$}}
\put(5.5,-4){\makebox(0,0)[l]{$2$}}
\put(13.5,-4){\makebox(0,0)[l]{$4$}}
\put(20,-4){\makebox(0,0)[l]{$6$}}
\put(27,-4){\makebox(0,0)[l]{$4$}}
\put(34,-4){\makebox(0,0)[l]{$2$}}
\put(9.5,7){\makebox(0,0)[l]{$1$}}
 \put(17,7){\makebox(0,0)[l]{$4$}}\put(30,7){\makebox(0,0)[l]{$2$}}
 \put(-12,0){\makebox(0,0)[l]{$11)$}}
\end{picture}\\[-23pt]

\setlength{\unitlength}{0.7mm}
\hfill\begin{picture}(100,0)(-10,-1)\footnotesize
\put(-35,0.5){\makebox(0,0)[l]{}}
\multiput(0,0)(7,0){1}{\circle{1.5}}\multiput(7,0)(7,0){1}{\circle*{1.5}}
\multiput(7,7)(0,7){1}{\circle{1.5}}\multiput(7,1)(0,7){1}{\line(0,1){5}}
\multiput(14,0)(7,0){6}{\circle{1.5}}
 \multiput(1,0)(7,0){7}{\line(1,0){5}}\multiput(28,7)(0,7){1}{\circle{1.5}}
 \multiput(28,1)(0,7){1}{\line(0,1){5}}
\put(-0.5,-4){\makebox(0,0)[l]{$1$}}
\put(5.5,-4){\makebox(0,0)[l]{$2$}}
\put(13.5,-4){\makebox(0,0)[l]{$4$}}
\put(20,-4){\makebox(0,0)[l]{$6$}}
\put(27,-4){\makebox(0,0)[l]{$8$}}
\put(34,-4){\makebox(0,0)[l]{$6$}}
\put(40.5,-4){\makebox(0,0)[l]{$4$}}
\put(48,-4){\makebox(0,0)[l]{$2$}}\put(9.5,7){\makebox(0,0)[l]{$1$}}
 \put(30,7){\makebox(0,0)[l]{$4$}}
 \put(-12,0){\makebox(0,0)[l]{$12)$}}
\end{picture}\\[-23pt]

\setlength{\unitlength}{0.7mm}
\hfill\begin{picture}(200,20)(-5,-1)\footnotesize
\put(-35,0.5){\makebox(0,0)[l]{}}
\multiput(0,0)(7,0){1}{\circle{1.5}}\multiput(7,0)(7,0){1}{\circle*{1.5}}
\multiput(7,7)(0,7){1}{\circle{1.5}}\multiput(7,1)(0,7){1}{\line(0,1){5}}
\multiput(14,0)(7,0){7}{\circle{1.5}}
 \multiput(1,0)(7,0){8}{\line(1,0){5}}\multiput(42,7)(0,7){1}{\circle{1.5}}
 \multiput(42,1)(0,7){1}{\line(0,1){5}}
\put(-0.5,-4){\makebox(0,0)[l]{$1$}}
\put(5.5,-4){\makebox(0,0)[l]{$2$}}
\put(13.5,-4){\makebox(0,0)[l]{$4$}}
\put(20,-4){\makebox(0,0)[l]{$6$}}
\put(27,-4){\makebox(0,0)[l]{$8$}}
\put(34,-4){\makebox(0,0)[l]{$10$}}
\put(40.5,-4){\makebox(0,0)[l]{$12$}}
\put(48,-4){\makebox(0,0)[l]{$8$}}\put(9.5,7){\makebox(0,0)[l]{$1$}}
\put(55,-4){\makebox(0,0)[l]{$4$}}\put(44,7){\makebox(0,0)[l]{$6$}}
\put(-12,0){\makebox(0,0)[l]{$13)$}}
\end{picture}\\[-23pt]

\setlength{\unitlength}{0.7mm}
\hfill\begin{picture}(100,0)(-26,-1)\footnotesize
\put(-35,-0.5){\makebox(0,0)[l]{}}
\multiput(-16,0)(7,0){2}{\circle{1.5}}\multiput(-15,0)(7,0){1}{\line(1,0){5}}
\multiput(-7.5,0)(1,0){7}{\circle*{0.5}}
 \multiput(0,0)(7,0){1}{\circle{1.5}}
 \multiput(7,0)(7,0){1}{\circle*{1.5}}
\multiput(-9,7)(0,7){1}{\circle{1.5}}\multiput(-9,1)(0,7){1}{\line(0,1){5}}
\multiput(14,0)(7,0){4}{\circle{1.5}}
 \multiput(1,0)(7,0){5}{\line(1,0){5}}
  \multiput(21,7)(0,7){1}{\circle{1.5}}
   \multiput(28,7)(0,7){1}{\circle{1.5}}
 \multiput(21,1)(0,7){1}{\line(0,1){5}}
  \multiput(22,7)(0,7){1}{\line(1,0){5}}
\put(-16.5,-4){\makebox(0,0)[l]{$1$}}
\put(-9.5,-4){\makebox(0,0)[l]{$2$}}
\put(-0.5,-4){\makebox(0,0)[l]{$2$}}
\put(6,-4){\makebox(0,0)[l]{$2$}}
\put(13.5,-4){\makebox(0,0)[l]{$4$}}
\put(20,-4){\makebox(0,0)[l]{$6$}}
\put(27,-4){\makebox(0,0)[l]{$4$}}
\put(34,-4){\makebox(0,0)[l]{$2$}}
\put(-7.5,7){\makebox(0,0)[l]{$1$}}
 \put(17,7){\makebox(0,0)[l]{$4$}}\put(30,7){\makebox(0,0)[l]{$2$}}
 \put(-28,0){\makebox(0,0)[l]{$14)$}}
\end{picture}\\[-23pt]

\setlength{\unitlength}{0.7mm}
\hfill\begin{picture}(200,20)(-21,-1)\footnotesize
\put(-35,0.5){\makebox(0,0)[l]{}}
\multiput(-16,0)(7,0){2}{\circle{1.5}}\multiput(-15,0)(7,0){1}{\line(1,0){5}}
\multiput(-7.5,0)(1,0){7}{\circle*{0.5}}
\multiput(0,0)(7,0){1}{\circle{1.5}}\multiput(7,0)(7,0){1}{\circle*{1.5}}
\multiput(-9,7)(0,7){1}{\circle{1.5}}\multiput(-9,1)(0,7){1}{\line(0,1){5}}
\multiput(14,0)(7,0){6}{\circle{1.5}}
 \multiput(1,0)(7,0){7}{\line(1,0){5}}\multiput(28,7)(0,7){1}{\circle{1.5}}
 \multiput(28,1)(0,7){1}{\line(0,1){5}}
\put(-16.5,-4){\makebox(0,0)[l]{$1$}}\put(-9.5,-4){\makebox(0,0)[l]{$2$}}
\put(-1 ,-4){\makebox(0,0)[l]{$2$}}
\put(6,-4){\makebox(0,0)[l]{$2$}}
\put(13.5,-4){\makebox(0,0)[l]{$4$}}
\put(20,-4){\makebox(0,0)[l]{$6$}}
\put(27,-4){\makebox(0,0)[l]{$8$}}
\put(34,-4){\makebox(0,0)[l]{$6$}}
\put(40.5,-4){\makebox(0,0)[l]{$4$}}
\put(48,-4){\makebox(0,0)[l]{$2$}}\put(-7.5,7){\makebox(0,0)[l]{$1$}}
 \put(30,7){\makebox(0,0)[l]{$4$}}
 \put(-28,0){\makebox(0,0)[l]{$15)$}}
\end{picture}\\[-23pt]

 \setlength{\unitlength}{0.7mm}
\hfill\begin{picture}(100,0)(-26,-1)\footnotesize
\put(-35,0.5){\makebox(0,0)[l]{}}
\multiput(-16,0)(7,0){2}{\circle{1.5}}\multiput(-15,0)(7,0){1}{\line(1,0){5}}
\multiput(-7.5,0)(1,0){7}{\circle*{0.5}}
\multiput(0,0)(7,0){1}{\circle{1.5}}
\multiput(7,0)(7,0){1}{\circle*{1.5}}
\multiput(-9,7)(0,7){1}{\circle{1.5}}\multiput(-9,1)(0,7){1}{\line(0,1){5}}
\multiput(14,0)(7,0){7}{\circle{1.5}}
 \multiput(1,0)(7,0){8}{\line(1,0){5}}\multiput(42,7)(0,7){1}{\circle{1.5}}
 \multiput(42,1)(0,7){1}{\line(0,1){5}}
\put(-16.5,-4){\makebox(0,0)[l]{$1$}}\put(-9.5,-4){\makebox(0,0)[l]{$2$}}
\put(-1 ,-4){\makebox(0,0)[l]{$2$}}
\put(6,-4){\makebox(0,0)[l]{$2$}}
\put(13.5,-4){\makebox(0,0)[l]{$4$}}
\put(20,-4){\makebox(0,0)[l]{$6$}}
\put(27,-4){\makebox(0,0)[l]{$8$}}
\put(34,-4){\makebox(0,0)[l]{$10$}}
\put(40.5,-4){\makebox(0,0)[l]{$12$}}
\put(48,-4){\makebox(0,0)[l]{$8$}}\put(-7.5,7){\makebox(0,0)[l]{$1$}}
\put(55,-4){\makebox(0,0)[l]{$4$}}\put(44,7){\makebox(0,0)[l]{$6$}}
\put(-28,0){\makebox(0,0)[l]{$16)$}}
\end{picture}\\[-23pt]

\setlength{\unitlength}{0.7mm}
\hfill\begin{picture}(200,20)(-5,-1)\footnotesize
\put(-35,-0.5){\makebox(0,0)[l]{}}
 \multiput(28,7)(0,7){1}{\circle*{1.5}}
 \multiput(35,7)(0,7){1}{\circle{1.5}}
 \multiput(28,1)(0,7){1}{\line(0,1){5}}
  \multiput(29,7)(0,7){1}{\line(1,0){5}}
\multiput(0,0)(7,0){9}{\circle{1.5}}
 \multiput(1,0)(7,0){8}{\line(1,0){5}}
\put(-0.5,-4){\makebox(0,0)[l]{$1$}}\put(6.5,-4){\makebox(0,0)[l]{$2$}}
\put(13.5,-4){\makebox(0,0)[l]{$3$}}
\put(20,-4){\makebox(0,0)[l]{$4$}}\put(56,-4){\makebox(0,0)[l]{$1$}}
\put(27,-4){\makebox(0,0)[l]{$5$}}\put(49,-4){\makebox(0,0)[l]{$2$}}
\put(34,-4){\makebox(0,0)[l]{$4$}}\put(42,-4){\makebox(0,0)[l]{$3$}}
 \put(24,7){\makebox(0,0)[l]{$2$}}\put(37,7){\makebox(0,0)[l]{$1$}}
 \put(-12,0){\makebox(0,0)[l]{$17)$}}
\end{picture}\\[-23pt]

\setlength{\unitlength}{0.7mm}
\hfill\begin{picture}(100,0)(-10,-1)\footnotesize
\put(-35,0.5){\makebox(0,0)[l]{}}
\multiput(0,0)(7,0){1}{\circle{1.5}}\multiput(7,0)(7,0){1}{\circle*{1.5}}
\multiput(7,7)(0,7){1}{\circle{1.5}}\multiput(7,1)(0,7){1}{\line(0,1){5}}
\multiput(14,0)(7,0){7}{\circle{1.5}}
 \multiput(1,0)(7,0){8}{\line(1,0){5}}\multiput(21,7)(0,7){1}{\circle{1.5}}
 \multiput(21,1)(0,7){1}{\line(0,1){5}}
\put(-0.5,-4){\makebox(0,0)[l]{$1$}}
\put(5.5,-4){\makebox(0,0)[l]{$2$}}
\put(13.5,-4){\makebox(0,0)[l]{$4$}}
\put(20,-4){\makebox(0,0)[l]{$6$}}
\put(27,-4){\makebox(0,0)[l]{$5$}}
\put(34,-4){\makebox(0,0)[l]{$4$}}
\put(40.5,-4){\makebox(0,0)[l]{$3$}}
\put(48,-4){\makebox(0,0)[l]{$2$}}\put(9.5,7){\makebox(0,0)[l]{$1$}}
\put(55,-4){\makebox(0,0)[l]{$1$}}\put(23,7){\makebox(0,0)[l]{$3$}}
\put(-12,0){\makebox(0,0)[l]{$18)$}}
\end{picture}\\[-23pt]

\setlength{\unitlength}{0.7mm}
\hfill\begin{picture}(200,20)(-21,-1)\footnotesize
\put(-35,0.5){\makebox(0,0)[l]{}}
\multiput(-16,0)(7,0){2}{\circle{1.5}}\multiput(-15,0)(7,0){1}{\line(1,0){5}}
\multiput(-7.5,0)(1,0){7}{\circle*{0.5}}
\multiput(0,0)(7,0){1}{\circle{1.5}}\multiput(7,0)(7,0){1}{\circle*{1.5}}
\multiput(-9,7)(0,7){1}{\circle{1.5}}\multiput(-9,1)(0,7){1}{\line(0,1){5}}
\multiput(14,0)(7,0){6}{\circle{1.5}}
 \multiput(1,0)(7,0){7}{\line(1,0){5}}\multiput(21,7)(0,7){1}{\circle{1.5}}
 \multiput(21,1)(0,7){1}{\line(0,1){5}}
\put(-16.5,-4){\makebox(0,0)[l]{$1$}}\put(-9.5,-4){\makebox(0,0)[l]{$2$}}
\put(-1 ,-4){\makebox(0,0)[l]{$2$}}
\put(6,-4){\makebox(0,0)[l]{$2$}}
\put(13.5,-4){\makebox(0,0)[l]{$4$}}
\put(20,-4){\makebox(0,0)[l]{$6$}}
\put(27,-4){\makebox(0,0)[l]{$5$}}
\put(34,-4){\makebox(0,0)[l]{$3$}}
\put(40.5,-4){\makebox(0,0)[l]{$2$}}
\put(48,-4){\makebox(0,0)[l]{$1$}}\put(-7.5,7){\makebox(0,0)[l]{$1$}}
 \put(23,7){\makebox(0,0)[l]{$3$}}
 \put(-28,0){\makebox(0,0)[l]{$19)$}}
\end{picture}\\[-23pt]

\setlength{\unitlength}{0.7mm}
\hfill\begin{picture}(100,0)(-10,-1)\footnotesize
\put(-35,0.5){\makebox(0,0)[l]{}}
\multiput(0,0)(7,0){6}{\circle{1.5}}\multiput(1,0)(7,0){12}{\line(1,0){5}}
\multiput(42,0)(1,0){1}{\circle*{1.5}}\multiput(49,0)(7,0){6}{\circle{1.5}}
\multiput(14,7)(0,7){1}{\circle{1.5}}\multiput(70,1)(0,7){1}{\line(0,1){5}}
\multiput(70,7)(0,7){1}{\circle{1.5}}\multiput(14,1)(0,7){1}{\line(0,1){5}}
\put(-0.5,-4){\makebox(0,0)[l]{$2$}}
\put(6.5,-4){\makebox(0,0)[l]{$4$}}
\put(13.5,-4){\makebox(0,0)[l]{$6$}}
\put(20,-4){\makebox(0,0)[l]{$5$}}
\put(27,-4){\makebox(0,0)[l]{$4$}}
\put(34,-4){\makebox(0,0)[l]{$3$}}
\put(40.5,-4){\makebox(0,0)[l]{$2$}}
\put(48,-4){\makebox(0,0)[l]{$3$}}\put(16,7){\makebox(0,0)[l]{$3$}}
 \put(55,-4){\makebox(0,0)[l]{$4$}} \put(62,-4){\makebox(0,0)[l]{$5$}}
 \put(69,-4){\makebox(0,0)[l]{$6$}} \put(76,-4){\makebox(0,0)[l]{$4$}}
\put(83,-4){\makebox(0,0)[l]{$2$}}\put(71.5,7){\makebox(0,0)[l]{$3$}}
\put(-12,0){\makebox(0,0)[l]{$20)$}}
\end{picture}\\[-23pt]

 \setlength{\unitlength}{0.7mm}
\hfill\begin{picture}(200,20)(-5,-1)\footnotesize
\put(-35,0.5){\makebox(0,0)[l]{}}
\multiput(0,0)(7,0){3}{\circle*{1.5}}\multiput(21,0)(7,0){2}{\circle{1.5}}
 \multiput(1,0)(7,0){4}{\line(1,0){5}}
\put(7,7){\circle{1.5}}\put(7,1){\line(0,1){5}}
\put(-0.5,-4){\makebox(0,0)[l]{$2$}}
\put(-4.5,3){\makebox(0,0)[l]{$-5$}}
\put(5.5,-4){\makebox(0,0)[l]{$10$}}
\put(1.5,3){\makebox(0,0)[l]{$-1$}}
\put(13.5,-4){\makebox(0,0)[l]{$3$}}
\put(9.5,3){\makebox(0,0)[l]{$-4$}}
\put(19.8,-4){\makebox(0,0)[l]{$2$}}
\put(27,-4){\makebox(0,0)[l]{$1$}}
\put(8.8,7){\makebox(0,0)[l]{$5$}}
\put(-12,0){\makebox(0,0)[l]{$21)$}}
\end{picture}\\[-23pt]

\setlength{\unitlength}{0.7mm}
\hfill\begin{picture}(100,0)(-10,-1)\footnotesize
\put(-35,0.5){\makebox(0,0)[l]{}}
\multiput(0,0)(7,0){6}{\circle{1.5}}\multiput(1,0)(7,0){11}{\line(1,0){5}}
\multiput(42,0)(1,0){1}{\circle*{1.5}}\multiput(49,0)(7,0){5}{\circle{1.5}}
\multiput(14,7)(0,7){1}{\circle{1.5}}\multiput(56,1)(0,7){1}{\line(0,1){5}}
\multiput(56,7)(0,7){1}{\circle{1.5}}\multiput(14,1)(0,7){1}{\line(0,1){5}}
\put(-0.5,-4){\makebox(0,0)[l]{$2$}}
\put(6.5,-4){\makebox(0,0)[l]{$4$}}
\put(13.5,-4){\makebox(0,0)[l]{$6$}}
\put(20,-4){\makebox(0,0)[l]{$5$}}
\put(27,-4){\makebox(0,0)[l]{$4$}}
\put(34,-4){\makebox(0,0)[l]{$3$}}
\put(40.5,-4){\makebox(0,0)[l]{$2$}}
\put(48,-4){\makebox(0,0)[l]{$3$}}\put(16,7){\makebox(0,0)[l]{$3$}}
 \put(55,-4){\makebox(0,0)[l]{$4$}} \put(62,-4){\makebox(0,0)[l]{$3$}}
 \put(69,-4){\makebox(0,0)[l]{$2$}} \put(76,-4){\makebox(0,0)[l]{$1$}}
\put(58,7){\makebox(0,0)[l]{$2$}}
\put(-12,0){\makebox(0,0)[l]{$22)$}}
\end{picture}\\[5pt]

See \S~5.5 for the Chern numbers of these 22 fibers.

\begin{theorem}\label{THMnew1.4} Assume that $g\geq 2$.
  If $ 2c_2(F)-c_1^2(F)<6$, then either $F=nC$ for some smooth curve $C$, or
   $F_{\rm red}$ admits at most one singular point $p$ other than nodes.
   One of the following cases occurs.

\bigskip {\rm I)} \ $F=nF_{\rm red}$.
  \begin{enumerate}
\item[$1)$] \ $F_{\rm red}$ is a smooth or nodal
curve. 
\item[$2)$] \ $p$ is of type
$A_2$. 
\item[$3)$] \   $p$ is of type
$A_3$ and any $(-2)$-curve does not pass through $p$.
\item[$4)$] \  $p$ is of type
$A_3$ and one $(-2)$-curve passes through $p$.

\item[$5)$] \  $p$ is of
type $D_4$. 
\end{enumerate}

{\rm II)} \ $F=nA+2nB$,  $A$ and $B$ are reduced nodal curves
without common components,  $AB=2$, $A^2=-4$ and $B^2=-1$. $A$ has
at most two connected components $A_1$ and $A_2$.
  \begin{enumerate}
\item[$6)$] \  $A\cap B=\{p,q\}$ and any $(-2)$-curve is not a connected component of $A$.

\item[$7)$] \   $A$ has two connected components and one is a $(-2)$-curve.

\item[$8)$] \   $A$ and $B$ are tangent at a point $p$.
\end{enumerate}

The invariants of these fibers $F$ are as follows, where $0\leq
N=g(F)-p_a(F_{\rm red})\leq g$.
 {\center\rm
\begin{tabular}{|c|c|c|c|c|c|c|c|c|c|c|c|c|c|c|}\hline
 \phantom{\footnotesize$\dfrac1{1}$} $F$\phantom{\footnotesize$\dfrac1{1}$} & 1 & 2
  & 3 &4 & 5 & 6
   & 7 & 8 \\
   \hline
 \phantom{$\dfrac11$}$2c_2-c_1^2$\phantom{$\dfrac11$} & $0$ &$\frac{7}{2}$  &$\frac92$  & $\frac{21}{4}$& $5$ &
 $3$ &$\frac{9}{2}$   &$\frac{11}{2}$
 \\ \hline
  \phantom{\footnotesize$\dfrac11$}$c_1^2-4N$\phantom{\footnotesize$\dfrac11$} & $0$ &
  $\frac16$
  & $\frac12$ &$\frac14$ & $1$ & $-1$
   & $-\frac32$ & $-\frac12$ \\
  \hline

  \phantom{$\dfrac11$}$c_2-2N$\phantom{$\dfrac11$} &$0$
  & $\frac{11}{6}$ & $\frac52$ &$\frac{11}4$  &$3$  &$1$&$\frac{3}{2}$  & $\frac52$\\ \hline

  \phantom{$\dfrac11$}$\chi-\frac12N$\phantom{$\dfrac11$} & 0 & $\frac16$ & $\frac14$ &$\frac14$ & $\frac{1}{3}$& 0
   & 0  &$\frac16$ \\ \hline

 \end{tabular}\\[0.3cm]}
\end{theorem}

Note that $2c_2- c_1^2< 6$ is equivalent to $8\chi- c_1^2<2$. Hence
the fibers satisfying $c_2(F)\leq 3$ or $\chi\leq\frac14$ are
included in the classification list $1)\sim 8)$. For a
non-semistable fiber, $c_1^2$, $c_2$ and $\chi$ are positive.
Therefore, one can check that $\frac{11}6$ (resp. $\frac16$) is the
lower bound of $c_2$ (resp. $\chi$) for non-semistable fibers. All
of the fibers from 2) to 8) can not be the fibers in an isotrivial
family of curves, because their semistable models are not smooth.

\begin{corollary} Let $s$ be the number of singular fibers of $f:X\to C$
and $g\geq 2$.

 {\rm 1)} If  $f$ is non-trivial, then $\chi_f\leq \frac{g}{2}\left(2g(C)-2+\frac{8}{3}s\right)$.

  {\rm 2)} If $f$ is isotrivial, then
$K_f^2\leq \left(4g-\frac{24}{5}\right)s,$ and $\chi_f\leq
\frac{5gs}{6}$.
\end{corollary}

As an application of Theorem \ref{THMnew1.4}, we have
\begin{corollary} Assume $f:X\to C$ is isotrivial.  Let $s$ be the number of singular fibers that
 are not multiples of a smooth curve. Then
$K_X^2\leq 8\chi(\mathcal{O}_X)-2s. $
\end{corollary}

This gives a new proof of Polizzi's theorem that $K_X^2\neq
8\chi(\mathcal{O}_X)-1$ when $f:X\to C$ is isotrivial \cite{Po08}.
We will give some other applications of the main results in each
section.

\section{Dual models $F^*$ of a fiber $F$}

We recall several models of a singular fiber in this section,
including the minimal model, normal crossing model, $n$-th root
model, semistable model, and the dual model.
\subsection{Normal crossing model.}
A curve $B$ on $X$ is a nonzero effective divisor.

\begin{definition}  A {\it partial
resolution} of the singularities of $B$ is a sequence of blowing-ups
$\sigma=\sigma_1\circ\sigma_2\circ \cdots \circ\sigma_r:$ $ \bar X
\to X$
$$
 (\bar
X,\sigma^* B)=(X_r,B_r)\overset{\sigma_r}\longrightarrow X_{r-1}
\overset{\sigma_{r-1}}\longrightarrow \cdots
\overset{\sigma_2}\longrightarrow (X_1,B_1)
\overset{\sigma_1}\longrightarrow (X_0,B_0)=(X,B),
$$
satisfying the following conditions:

(i) \ $B_{r,\textrm{red}}$ has at worst ordinary double points as
its singularities.

(ii) \,$B_i=\sigma_i^*B_{i-1}$ is the total transform of $B_{i-1}$.

\smallskip \noindent Furthermore,  $\sigma$ is called the {\it
minimal partial resolution} of the singularities of $B$ if
\smallskip

(iii) $\sigma_i$ is the blowing-up of $X_{i-1}$ at a singular point
$(B_{i-1,\textrm{red}},p_{i-1})$ which is not an ordinary double
point for any $i\leq r$.
\end{definition}

The {\it minimal model} of $F$ is obtained by contracting all
$(-1)$-curves in $F$. Denote by $\bar F$ the partial resolution
 of the singularities of the minimal model of $F$.

\begin{definition}  $\bar F$ is called the {\it normal crossing model} of
$F$. If $\sigma$ is minimal, then we say that $\bar F$ is the {\it
minimal normal crossing model}.
\end{definition}

A $(-1)$-curve in $\bar F$ is called {\it redundant} if it meets the
other components in at most two points. It is obvious that a
redundant $(-1)$-curve can be contracted without introducing
singularities worse than ordinary double points. The minimal normal
crossing model of $F$ contains no redundant $(-1)$-curves, and it
can be obtained from any normal crossing model by contracting all
redundant $(-1)$-curves. In fact, the minimal normal crossing model
of $F$ is determined uniquely by $F$.

\subsection{$n$-th root model and the semistable model of $F$.}
 Let $\pi:\widetilde C\to C$ be a base change of degree $n$. Then we
 can construct the pullback fibration $\widetilde f:\widetilde X\to
 \widetilde C$ of $f:X\to C$ as follows.

$$\xymatrix{
  \widetilde X \ar[drr]_{\widetilde f}  &\ar[l]_{\tau} X'{\ar@/{^}/[rr]^\Pi} \ar[dr]^{f'} \ar[r]_{\pi_2}
   & X_1 \ar[d]^{f_1}
  \ar[r]_{\hskip0.4cm \pi_1} & X \ar[d]^{f} \\
   &   &  \widetilde C \ar[r]^{\pi} & C  }
$$
where $X_1=X\times_C\widetilde C$, $\pi_1$ and $f_1$ are the
projections. $X'$ is the minimal resolution of the singularities of
the normalization of $X_1$. $\tau$ is the contraction of those
$(-1)$-curves in the fibers. Then we get the pullback fibration
$\widetilde f$ of $f$ under the base change $\pi$.

Now we consider the above construction locally. Let $F$ be a fiber
of $f$ over $p\in C$. Assume that $\pi$ is totally ramified over
$p$, i.e., $\pi^{-1}(p)$ contains only one point $\widetilde p$. In
this case, $\pi$ is defined locally by $z=w^n$ near $p=0$.

Now denote by $\widetilde F$ (resp. $F'$) the fiber of $\widetilde
f$ (resp. $f'$) over $\widetilde p\in \widetilde C$. In fact,
$F'=\frac1n\Pi^*(F)$.
\begin{definition} The fiber $\widetilde F$ of $\widetilde f$ over $\widetilde p$
is called the {\it $n$-th root model} of $F$.
\end{definition}

Note that $F$ and any of its normal crossing model $\bar F$ have the
same $n$-th root model $\widetilde F$ for any $n$. In fact, if $F$
is normal crossing, then $F'$ is also normal crossing. In
particular, $\bar F'$ is the normal crossing model of $\widetilde
F$.

Indeed, we can assume that $F=\bar F=\sum_{i=1}^k n_iC_i$ is normal
crossing, where $C_i$ is irreducible. Let $p$ be a singular point of
$ F_{\rm red}$. Without loss of generality, we assume that $p$ is an
intersection point of $C_i$ with $C_j$. Near $p$, $\pi_1$ is defined
locally by $z^n=x^{n_i}y^{n_j}$. Then we see that the singularities
of the normalization of $X_1$ are of Hirzebruch-Jung type. Hence,
$F'$ is normal crossing. By the computation of the normalization, we
see that the multiplicity of the strict transform of $C_i$ in $F'$
is $n_i/\gcd(n,n_i)$.

If $n_i$ divides $n$ for any $i$, then one can prove that $F'$ and
$\widetilde F$ are semistable. This is the famous Semistable
Reduction Theorem.  Denote by $M_F=\textrm{lcm}\{n_1,\cdots,n_k\}$.
Then the $n$-th root model of $F$ is always semistable for any $n$
satisfying $n\equiv 0\pmod{M_F}$.

\begin{definition} If $\widetilde F$ is semistable, then
$\widetilde F$ is called the {\it semistable model} of $F$, or the
{\it semistable reduction} of $F$.
\end{definition}

\subsection{Dual model $F^*$ of $F$}\label{Sect2.3}
\begin{definition} If $n\equiv -1 \pmod{M_F}$, then the $n$-th root model
of $F$ is called the {\it dual model} of $F$, denoted by $F^*$.
\end{definition}
The dual model is introduced first by Kodaira for elliptic
fibrations. Our definition is a natural generalization. In general,
$(F^*)^*$ doesn't coincide with $F$ unless
 the semistable model of $F$ is smooth. (If the uniqueness of the dual model is
 needed,  one may choose $n$ to be the minimal positive integer
 satisfying $n\equiv -1 \pmod{M_F}$).

Let ${\bar{F}}=\sum_{i=1}^{k}n_iC_i$ be the minimal normal crossing
model of $F$, where $C_i$'s are all irreducible components. We have
seen that $\bar F'$ is the normal crossing model of $F^*$.

Let $n\equiv-1\pmod{M_F}$. Denote by $C_i^*$ the strict transform of
$C_i$ in $\bar F'$. Because $n_i$ is prime to $n$ for any $i$,
$C_i^*$ is irreducible. The multiplicity of $C_i^*$ in $\bar F'$ is
still $n_i$. By the resolution of Hirzebruch-Jung singularities, we
see that $\bar F'$ is obtained by inserting a chain of rational
curves.
$$
\bar F'=\sum_{i=1}^kn_iC_i^*+\sum_p\Gamma_p^*,
$$
where $p$ runs over all double points of $\bar F$,
$\Gamma_p^*=\sum_{i=1}^r\gamma_i\Gamma_i$. Assume that $p$ is an
intersection point of two local components $C_i$ and $C_j$. Then
near $\Gamma_p^*$, $\bar F'$ is as follows, where $\gamma_0=n_i$ and
$\gamma_{r+1}=n_j$.

\setlength{\unitlength}{0.9mm} \hfill\begin{picture}(168,12)(-50,-1)
\put(-35,0.5){\makebox(0,0)[l]{}} \put(0,0){\circle*{1.5}}
\multiput(7,0)(7,0){2}{\circle{1.5}}\multiput(1,0)(7,0){3}{\line(1,0){5}}
\multiput(21,0)(1,0){6}{\circle*{0.5}}\multiput(28,0)(7,0){3}{\line(1,0){5}}
\multiput(34,0)(7,0){2}{\circle{1.5}}\put(48,0){\circle*{1.5}}
\put(-1,4){\makebox(0,0)[l]{$C_i^*$}}\put(6.5,4){\makebox(0,0)[l]{$\Gamma_1$}}
\put(13.5,4){\makebox(0,0)[l]{$\Gamma_2$}}\put(30,4){\makebox(0,0)[l]{$\Gamma_{r-1}$}}
\put(41,4){\makebox(0,0)[l]{$\Gamma_r$}}\put(49,4){\makebox(0,0)[l]{$C_j^*$}}
\put(-10,-4){\makebox(0,0)[l]{$n_i=\gamma_0$}}\put(6.2,-4){\makebox(0,0)[l]{$\gamma_1$}}\put(13,-4)
{\makebox(0,0)[l]{$\gamma_2$}}
\put(31,-4){\makebox(0,0)[l]{$\gamma_{r-1}$}}\put(40,-4){\makebox(0,0)[l]{$\gamma_{r}$}}\put(46,-4)
{\makebox(0,0)[l]{$\gamma_{r+1}=n_j$}}
\end{picture}\\

\begin{lemma}\label{gamma} {\rm 1)} \ For $i=1,\cdots,r$, we have
$\gamma_i\,|\,\gamma_{i-1}+\gamma_{i+1}$.

{\rm 2)} \ $\gamma_0\,|\,\gamma_1+\gamma_{r+1}$ and
$\gamma_{r+1}\,|\,\gamma_r+\gamma_0$.
\end{lemma}
\begin{proof} The local base change over $p$ is defined by
$z^{n}=x^{n_i}y^{n_j}$.  Note that $n$ is prime to $n_i$ and $n_j$,
the equation is equivalent to $z^n=xy^{n-q}$ for some $q$ satisfying
$n_j+qn_i\equiv 0\pmod{n}$ and $1\leq q <n$ (see \cite{BPV},
Ch.~III, \S5). By definition, $n_i$ divides $n+1$. One can see that
$q_0=-(n+1)n_j/n_i=-(n+1)\gamma_{r+1}/\gamma_0$ is an integer
satisfying $q\equiv q_0\pmod n$. The singular point over $p$ is of
Hirzebruch-Jung type.

For convenience, we take $\Gamma_0=C_i^*$, $\gamma_0=n_i$,
$\Gamma_{r+1}=C_j^*$ and $\gamma_{r+1}=n_j$. Let $e_i=-\Gamma_i^2$.
By Zariski's lemma (\cite{BPV}, Ch.~III, \S8), $\bar F'\cdot
\Gamma_i=0$ for $i=1,\cdots,r$, thus we have
\begin{equation}\label{LinearEquation}
\begin{cases}
-\gamma_0+\gamma_1e_1-\gamma_2=0, &\\
-\gamma_1+\gamma_2e_2-\gamma_3=0,&\\
\hskip1.5cm \vdots & \\
-\gamma_{r-1}+\gamma_re_r-\gamma_{r+1}=0.&
\end{cases}
\end{equation}
So we have proved 1). For fixed $\gamma_0$ and $\gamma_{r+1}$, this
is a linear system of the $r$ variables $\gamma_1$, $\cdots$,
$\gamma_{r}$. We denote by $A=[e_1,\cdots,e_r]$ the coefficient
matrix. It is well-known that the determinant of $A$ is equal to
$n$, and the determinant of the submatrix $[e_2,\cdots,e_r]$ is
equal to $q$. By Gramer Rule,
$$
\gamma_1=\frac{\gamma_0q+\gamma_{r+1}}{n}=\frac{\gamma_0q_0+
\gamma_{r+1}}{n}+\gamma_0\dfrac{q-q_0}n=-\gamma_{r+1}+\gamma_0\dfrac{q-q_0}n,
$$
so $\gamma_0\,|\,\gamma_1+\gamma_{r+1}$. Symmetrically,
$\gamma_{r+1}\,|\,\gamma_r+\gamma_0$.
 \end{proof}

\begin{lemma}\label{pa} The reduced normal crossing models of $F$ and $F^*$ have the
same arithmetic genus, i.e., $p_a(\bar F_{\rm red})=p_a({\bar
{F'}}_{\rm red})$.
\end{lemma}
\begin{proof} This follows from the fact that the arithmetic genus of $\bar F$ is equal to
the sum of the geometric genus of each component plus the number of
cycles in the dual graph of $\bar F$. Note that the geometric genera
of $C_i$ and $C_i^*$ are the same. So insert a Hirzebruch-Jung chain
of rational curves does not change the arithmetic genus.
\end{proof}

\section{Local invariants of a fiber}
In order to obtain the computation formulas for the Chern numbers of
a singular fiber, we need to introduce several local invariants for
a singular point of a curve, not necessarily reduced. See
\cite{Ta96}.
\subsection{Invariants $\alpha$ and $\beta$ for a curve singularity}
In Definition 2.1, we denote by $m_{i+1}$ the multiplicity of
$(B_{i,\textrm{red}},p_{i})$ at $p_i$. (Note that
$B_{i,\textrm{red}}$ is the reduced {\it total} transform  of
$B_{\rm red}$, instead of the {\it strict} transform). One can check
that if $B$ is a compact curve, then
\begin{equation}\label{ArithGenus}
p_a(B_{r, \rm red})=p_a(B_{\rm
red})-\dfrac12\sum_{i=1}^{r}(m_i-1)(m_i-2).
\end{equation}

Suppose $B$ has only one singular point $p=p_0$. Let $k_p=k_p(B)$
(resp. $\mu_p=\mu_p(B)$) be the number of local branches (resp.
Milnor number) of $(B_{\rm red},p)$. Then
\begin{equation}
\mu_p=\sum_{i=1}^r(m_i-1)(m_i-2)+k_p-1.
\end{equation}

1) \ $m_i =2$ for all $i$ if and only if $(B_{\rm red}, p)$ is a
node.

2) \ $m_i\leq 3$ for all $i$ if and only if $(B_{\rm red}, p)$ is an
$ADE$ singular point (\cite{BPV}, Ch.II, \S8).

\bigskip If $q\in B_{r,\textrm{red}}$ is a double point, and the two
local components of $(B_r,q)$ have multiplicities $a_q$ and $b_q$,
then we define $[a_q,b_q]:=\frac{\gcd(a_q,b_q)^2}{a_qb_q},$ and
\begin{equation}
\alpha_p=\sum_{i=1}^r(m_i-2)^2, \hskip1cm \beta_p=\sum_{q\in
B_r}[a_q,b_q],
\end{equation}
where $q$ runs over all of the double points of
$B_{r,\textrm{red}}$. These two invariants are independent of the
 resolution.

 In \cite{Ta96}, we prove that $\mu_p\geq
\alpha_p+\beta_p$. Actually, we need more precise inequality of this
kind.

\begin{example} The invariants of an $ADE$ singularity $(B_{\rm red},p)$ are as
follows.
\end{example}
\vskip0.2cm \centerline{
\begin{tabular}{|c|c|c|c|c|c|c|c|c|c|c|c|c|c|}\hline
   & $A_{2k-1}$ &
  $A_{2k}$
  & $D_{2k+2}$ & $\phantom{\dfrac11}D_{2k+3}\phantom{\dfrac11}$ & $E_6$ & $E_7$
   & $E_8$    \\
  \hline
  $\mu_p$ & $2k-1$ & $2k$ & $2k+2$ & $\phantom{\dfrac11}2k+3\phantom{\dfrac11}$ & $6$ & $7$
  & $8$
   \\ \hline
  $\alpha_p$ & $k-1$ & $k$ & $k$ & $\phantom{\dfrac11}k+1\phantom{\dfrac11}$ & $3$ & $3$
  & $4$
   \\ \hline
    $\beta_p$ & I$_k$ &
    $\frac{3k}{2k+1}$ &
      II$_k$
    & III$_k$ & $\phantom{\dfrac11}1\phantom{\dfrac11}$ & IV
  & $\frac45$
   \\ \hline
    $\beta^-_p$ & $\geq 1-\frac1k$ & $\geq\frac{6k-1}{4k+2}$ & $ \ $ & $\phantom{\dfrac11}\geq \frac12\phantom{\dfrac11}$ &
    $\geq\frac{11}{12}$ & $\geq\frac{1}{3}$
  & $\geq\frac{11}{15}$
   \\ \hline
 \end{tabular}\\[0.5cm]}

$$\begin{cases} {\rm I}_k=1-\frac{1}{k}+[k(n+m), \ n]+[k(n+m), \ m].
&\\
{\rm II}_k=\frac{k(n,m+l)^2}{n(n+k(m+l))}+[n+k(m+l), \ m]+[n+k(m+l),
\ l].
&\\
{\rm III}_k=\frac{1}{2}+ [m, \ 2((2k+1)m+n)]+
\frac{(2k+1)(n,2m)^2}{2n((2k+1)m+n)}.&\\
\text{\rm IV}=\frac{1}{3}+\frac{2(3m,n)^2}{3n(2m+n)}+\frac{(m,3n)^2}
{3m(2m+n)}.&
\end{cases}
$$
{Where $n$ (resp. $m$ or $l$) is the multiplicity of a local branch
of $(F,p)$.  $n$ corresponds to a smooth branch. We have
$$
{\rm I}_k\leq 1, \hskip0.3cm {\rm II}_k\leq1, \hskip0.3cm {\rm
III}_k\leq\frac{3(k+1)}{2k+3}, \hskip0.3cm {\rm IV}\leq\frac45.
$$
 }

\begin{lemma}\label{mualphabeta} {\rm 1)} \ $\mu_p\geq \alpha_p+\beta_p$, with
equality iff the singularity is of types $A_1$ or $A_2$.

{\rm 2)} \ $\mu_p\geq \alpha_p+\beta_p+1$ except for the
singularities of types $A_k$ for $k\leq 4$.

{\rm 3)} \ $\mu_p\geq \alpha_p+\beta_p+2$ except for the
singularities of types $A_k$ $(k\leq 6)$ and $D_5$.

{\rm 4)} \ If $2(\mu_p-\alpha_p-\beta_p)+\alpha_p+3\beta_p^-<6$,
then $p$ is of types $A_1$, $A_2$, $A_3$ and $D_4$.

{\ } \hskip0.33cm  If
   $2(\mu_p-\alpha_p-\beta_p)+\alpha_p+3\beta_p^-<5$, then $p$ is
   of types $A_1$, $A_2$ and $A_3$.

{\ } \hskip0.33cm  If
   $2(\mu_p-\alpha_p-\beta_p)+\alpha_p+3\beta_p^-<\frac72$, then $p$ is a
node.
\end{lemma}
\begin{proof} For an $ADE$ singular point $p$, the inequalities can
be checked directly from the computation above.

If $p$ is not an $ADE$ singular point, then at least one $m_i\geq
4$, so $\alpha_p\geq 4$. We claim that  $\mu_p\geq
\alpha_p+\beta_p+2$. In Definition 2.1, we assume that
$\sigma=\sigma_1\circ\sigma_2$, where $\sigma_1:X'\to X$ consists of
blowing-ups at the non-$ADE$ singular points $p_0$, $\cdots$,
$p_{r'-1}$ such that $B'=\sigma_1^*B$ admits at worst $ADE$ singular
points. Then we have
\begin{equation}
\mu_p-\alpha_p-\beta_p=\sum_{i=1}^{r'}(m_i-3)+\sum_{p'\in
B'}(\mu_{p'}-\alpha_{p'}-\beta_{p'})
\end{equation}
Because $p$ is not an $ADE$ singular point, at least one of $m_i$
($i\leq r'$) is bigger than $3$. If two of these $m_i$'s are bigger
than 3, then $\mu_p\geq \alpha_p+\beta_p+2$.  Without loss of
generality, we assume that $m_1=4$ and $r'=1$. Namely $m_1=4$ and
$m_i\leq 3$ for all $i\geq 2$. We can assume also that
$\mu_{p'}<\alpha_{p'}+\beta_{p'}+1$ for any singular point $p'$ of
$B'_{\rm red}$.

Now we consider the $ADE$ singular points of $B'$. Because the
exceptional curve is one of the branches of the singular points $p'$
of $B'_{\rm red}$, each singular point $p'$ has at least two
branches. According to 1), the singular points $p'$ of $B'_{\rm
red}$ is of types $A_1$ or $A_3$. Note that if $p'$ is of type
$A_3$, then $\mu_{p'}-\alpha_{p'}-\beta_{p'}=\frac12$. Thus if $B'$
admits at least two $A_3$, then $\mu_p\geq \alpha_p+\beta_p+2$ holds
true.

If $B'$ admits only one $A_3$, then we can assume that $(B,p)$ is
defined by $(x-y)^a(x+y)^b(x^2-y^3)^c=0$. Now it is easy to check
that $\mu_p=10$, $\alpha_p=5$ and $\beta_p\leq 2$. So $\mu_p\geq
\alpha_p+\beta_p+2$.

If $B'$ admits no $A_3$, then $B'$ admits 4 $A_1$. Hence we can
assume that $(B,p)$ is defined by $x^ay^b(x-y)^c(x+y)^d=0$. We have
$\mu_p=9$, $\alpha_p=4$ and $\beta_p\leq 1$. Thus  $\mu_p\geq
\alpha_p+\beta_p+2$.
\end{proof}

\begin{lemma}\label{ADE5} A curve singularity $p$ satisfying $\sum_{i=1}^rm_i(m_i-2)\leq
5$ must be of types $A_1$, $A_2$, $A_3$ and $D_4$.
\end{lemma}
\begin{proof}
The condition implies that $m_i\leq 3$ for any $i$ and there exists
at most one $i$ such that $m_i=3$, so $p$ is an $ADE$ singular
point. Now one can check the result directly.
\end{proof}

We define $\beta_F$ as the sum of $\beta_p$. One can check easily
that $\beta_F$ is independent of the resolution,  thus  $F$,
$\sigma^*F$ and $\bar F$ have the same $\beta$-invariants.

\subsection{Invariants $\beta^-$ and $\beta^+$}\label{Sect3.2}

\begin{definition} Let $\bar F$ be the minimal normal crossing model
of $F$, and let $G(\bar F)$ be the dual graph of $\bar F$. A H-J
branch of rational curves in $G(\bar F)$ is
\begin{equation*}
\overset{\gamma_1}{\underset{-e_1}\circ}\!\!\!\!\!-\!\!\!-\!\!\!-\!\!\!-\!\!\!-\!\!\!\!\!
\overset{\gamma_2}{\underset{-e_2}\circ}\!\!\!\!\!-\!\!\!-\!\!\!-\!\!\!-\!\!\!-
\ \cdots \ -\!\!\!-\!\!\!-\!\!\!-\!\!\!-\!\!\!\!\!
\overset{\gamma_r}{\underset{-e_r}\circ}
\!\!\!\!\!-\!\!\!-\!\!\!-\!\!\!-\!\!\!-\!\!\!\!\!\!
\overset{\gamma_{r+1}}\bullet
\end{equation*}
where $\overset{\gamma_i}{\underset{-e_i}\circ}$ denotes a smooth
rational curve $\Gamma_i$ with $\Gamma_i^2=-e_i$ whose multiplicity
in $\bar F$ is $\gamma_i$. $\bullet$ denotes either a curve
$\Gamma\not\cong\mathbb P^1$, or a smooth rational curve meeting at
$3$ or more points with the other components. We call $\Gamma_1$ an
{\it end point} of $G(\bar F)$.
\end{definition}

Note that the $r$ rational curves can be contracted to a
Hirzebruch-Jung singularity of type $(n,q)$ with defining equation
$z^n=xy^{n-q}$ (\cite{BPV}, Ch.~III, \S5), where $n$ and $q$ are
respectively the determinants of the matrices $[e_1,\cdots,e_r]$ and
$[e_2,\cdots,e_r]$.  $n$ and $q$ can also be determined by the
multiplicities $\gamma_i$ as follows.

According to (\ref{LinearEquation}) and $\gamma_0=0$, we see that
$\gamma_1$ divides $\gamma_i$ for any $i$. Using the notations of
(\cite{BPV}, Ch.~III, \S5), $\gamma_1<\gamma_2<\cdots<\gamma_r$,
 $\gamma_i=\mu_i \gamma_1$ for any $i$, so
$1=\mu_1<\mu_2<\cdots<\mu_{r+1}$.
$$n=\mu_{r+1}=\dfrac{\gamma_{r+1}}{\gamma_1},\hskip0.3cm q'=\mu_r=\dfrac{\gamma_r}{\gamma_1}$$
and $q$ is the unique solution of the equation
$$qq'\equiv 1\pmod{n}, \hskip0.3cm 1\leq q<n.$$
 Since $\mu_i$ and $\mu_{i+1}$ are coprime, the contribution of the branch
 to $\beta_F=\beta_{\bar F}$ is
\begin{equation}
\beta'=\dfrac1{\mu_1\mu_2}+\dfrac1{\mu_2\mu_3}+\cdots+\dfrac1{\mu_r\mu_{r+1}}.
\end{equation}
There is a relation (\cite{BPV}, Ch. III, \S5, eq(6))
 \begin{equation}
   \lambda_k\mu_{k+1}-\lambda_{k+1}\mu_k=n,
 \end{equation}
i.e.,
 \begin{equation}\label{1.6}
   \dfrac{\lambda_k}{\mu_{k}}-\dfrac{\lambda_{k+1}}{\mu_{k+1}}
   =n\dfrac1{\mu_k\mu_{k+1}}.
 \end{equation}
Note that $\lambda_1=q$ and $\lambda_{r+1}=0$. Take the sum of
(\ref{1.6}) from $k=1$ to $r$, we have
\begin{equation}\label{LuJun}
\beta'=\dfrac1n\left(\dfrac{\lambda_1}{\mu_1}-\dfrac{\lambda_{r+1}}
{\mu_{r+1}}\right)=\dfrac{q}n.
\end{equation}
\begin{lemma}  The contribution of the H-J branch to $\beta_F$
is $\frac{q}{n}$.
 \end{lemma}
\begin{definition} $\beta_F^-=\sum\beta'$ is the total contribution of
all H-J branches
 in $G(\bar F)$.
\end{definition}

Note that $\gamma_2=e_1\gamma_1$, the contribution of a H-J branch
to $\beta^-_F$ is at least $[\gamma_1,\gamma_2]=\frac1{e_1}$.

\begin{example} If $e_1=\cdots=e_{r-1}=2$ and $e_r\geq 2$, then $n=r(e_r-1)+1$, $q=n-(e_r-1)=(r-1)(e_r-1)+1$, and the
contribution of this H-J branch to $\beta^-_F$ is
\begin{equation}\label{beta-}\beta'=\dfrac{(r-1)(e_r-1)+1}{r(e_r-1)+1}=1-\dfrac{e_r-1}{r(e_r-1)+1}.\end{equation}
\end{example}

\begin{theorem} {\bf (Gang Xiao \cite{Xi90})} Assume that  $n\equiv 0\pmod{M_F}$.
Let $\bar F$ be the minimal normal crossing model
of $F$. Consider the construction of the $n$-th root model of $\bar
F$ as in \S2.2. Then a curve in $X'$ is contracted by $\tau$ if and
only if it comes from a H-J branch in $\bar F$.
\end{theorem}

The theorem above is contained in the proof of Prop. 1 of
\cite{Xi90}.

From the previous theorem,  $\beta_F^-$ is just $c_{-1}(F)$ defined
in \cite{Ta96} by the remark of (\cite{Ta96}, p.666), i.e.,
$n\beta_F^-$ is the number of $(-1)$-curves contracted by $\tau$.
Let $\beta^+_F=\beta_F-\beta^-_F$.
$$
 \beta_F=\beta_F^++\beta_F^-.
$$

\subsection{Formulas for the Chern numbers of a fiber}\label{S3.3}

Let $\mu_{F}=\sum_{p}\mu_p(F_{\rm red})$ be the sum of the Milnor
numbers of the singularities of $F_{\textrm{red}}$.

Let $N_F=g-p_a(F_{\textrm{red}})$. One can prove that $0\leq N_F\leq
g$. $N_F=0$ iff $F$ is reduced, or $g=1$ and $F$ is of type
$_m\textrm{I}_b$. $N_F=g$ iff $F$ is a tree of smooth rational
curves.

The topological characteristic of $F$ is equal to $2N_F+\mu_F+2-2g$.

Then we have the following formulas for the computation of the Chern
numbers of $F$.
\begin{equation}
\begin{cases}
c_1^2(F)=4N_F+F^2_{\textrm{red}}+\alpha_F-\beta_F^-,  & \\
c_2(F)=2N_F+\mu_F-\beta_F^+, & \\
12\chi_F=6N_F +F^2_{\textrm{red}}+\alpha_F+\mu_F-\beta_F. &
\end{cases}
\end{equation}

From the blow-up formulas, we only need to compute the Chern numbers
of the minimal normal crossing model $\bar F$.

\section{Proof of Theorem \ref{THMnew1.1}}

\subsection{Dedekind's reciprocity law} We denote by $(p,q)$ the
greatest common divisor of two integers $p$ and $q$.  The following
notation is from Dedekind's Reciprocity Law. Take
$$\chi(p,q)=\frac{1}{12}\left(\frac{q}{p}+\frac{p}{q}+\frac{(p,q)^2}{pq}\right)-\frac{1}{4}.$$
One can check easily the following identities
\begin{align}
\chi(p,p)=0, \hskip0.5cm  \chi(p,q)=\chi(p,p+q)+\chi(p+q,q).
\end{align}

If $p$ and $q$ are coprime, then Dedekind's sum  is defined as
follows
\begin{align*}
s(p,q)=\sum\limits_{i=0}^{q-1}\left(\!\!\left(\frac{pi}{q}\right)\!\!\right)\left(\!\!\left(\frac{i}{q}\right)\!\!\right),
\end{align*}
where
\begin{align*}
((x))= \left \{ \begin{array}{ll}
x-[x]-\frac{1}{2}, & x\notin \mathbb{Z},\\
0, & x\in \mathbb{Z},
\end{array}\right.
\end{align*}
 and $[x]$ is the largest integer   $\leq x$. $((x))$ is an odd
 fuction since $((-x))=-((x))$ and is periodic with period $1$.

 If $p$ and $q$ are not coprime, then we define $s(p,q):=s\left({p}/{(p,q)},
 {q}/{(p,q)}\right)$.
 Therefore, $s(-p,q)+s(p,q)=0$, and $s(p+kq,q)=s(p,q)$ for all
integers  $k$. In particular, if $p+p'$ is divisible by $q$, then
\begin{equation}
s(p,q)+s(p',q)=0.
\end{equation}

 The well-known {\itshape Dedekind's Reciprocity Law} says
\begin{align}
s(p,q)+s(q,p)=\chi(p,q)
\end{align}

\subsection{Compute $\chi_F$ from the normal crossing model $\bar F$}
Let $F$ be a singular fiber and ${\bar{F}}=\sum_{i=1}^{k}n_iC_i$ be
the normal crossing model of $F$, where $C_i$'s are all irreducible
components.   Take $M_F=\textrm{lcm}(n_1,\cdots,n_k)$.

\begin{theorem}\label{theorem_chi_F}
Let $N_{{\bar{F}}}=g-p_a({\bar{F}}_{\rm{red}})$. Then
  $$\chi_{F}=\frac{1}{2}N_{{\bar{F}}}-\sum_{i<j}\chi(n_i,n_j)C_iC_j.$$
\end{theorem}
\begin{proof}
 Note that $\chi_F$ is a birational invariant, so
$$\chi_F=\chi_{{\bar{F}}}=\frac{1}{2}N_{{\bar{F}}}+\frac{1}{12}(\mu_{{\bar{F}}}-\beta_{\bar F}+{\bar{F}}_{\rm{red}}^2).$$
By definition,
\begin{align*}
\mu_{{\bar{F}}}&=\sum\limits_{i<j}C_iC_j,\hskip0.3cm \beta_{\bar
F}=\sum\limits_{i<j} \frac{(n_i,n_j)^2}{n_in_j}C_iC_j, \hskip0.3cm
{\bar{F}}_{\rm{red}}^2=\sum\limits_{i<j}2C_iC_j+\sum\limits_{i=1}^{k}C_i^2.
\end{align*}
Since $C_i\bar F=0$,  $C_i^2=-\sum_{j\ne i}\frac{n_j}{n_i}C_iC_j$,
we have
$\sum_{i=1}^kC_i^2=-\sum_{i<j}\left(\frac{n_i}{n_j}+\frac{n_j}{n_i}\right)C_iC_j$.
Thus
\begin{align*}
{\mu_{{\bar{F}}}-\beta_{\bar F}+{\bar{F}}_{\rm{red}}^2 } &
=\sum\limits_{i<j}\left(3-\frac{(n_i,n_j)^2}{n_in_j}-\frac{n_j}{n_i}-\frac{n_i}{n_j}\right)C_iC_j
=-12\sum\limits_{i<j}\chi(n_i,n_j)C_iC_j.
\end{align*}
Hence
$\chi_{F}=\frac{1}{2}N_{{\bar{F}}}-\sum\limits_{i<j}\chi(n_i,n_j)C_iC_j$.
\end{proof}

\subsection{Duality theorem for $\chi$}

\begin{theorem}\label{theorem_chi_F} $F^{*}$ is the dual fiber of
$F$. Then
 $\chi_{F}+\chi_{F^{*}}=N_{{\bar{F}}}=N_{{\bar{F^*}}}$.
\end{theorem}
\begin{proof}  We use the notations in \S~\ref{Sect2.3}. We have seen that the normal crossing model $\bar{F^*}$ of $F^*$
is of the following type.
$$
\bar{F^*}=\sum_{i=1}^kn_iC_i^*+\sum_p\Gamma_p^*,
$$
where $p$ runs over all double points of $\bar F$, and
$\Gamma_p^*=\gamma_1\Gamma_1+\cdots+\gamma_r\Gamma_r$ is as follows,

 \setlength{\unitlength}{0.9mm}
\hfill\begin{picture}(168,12)(-50,-1)
\put(-35,0.5){\makebox(0,0)[l]{}} \put(0,0){\circle*{1.5}}
\multiput(7,0)(7,0){2}{\circle{1.5}}\multiput(1,0)(7,0){3}{\line(1,0){5}}
\multiput(21,0)(1,0){6}{\circle*{0.5}}\multiput(28,0)(7,0){3}{\line(1,0){5}}
\multiput(34,0)(7,0){2}{\circle{1.5}}\put(48,0){\circle*{1.5}}
\put(-1,4){\makebox(0,0)[l]{$C_i^*$}}\put(6,4){\makebox(0,0)[l]{$\Gamma_1$}}
\put(13.5,4){\makebox(0,0)[l]{$\Gamma_2$}}\put(30,4){\makebox(0,0)[l]{$\Gamma_{r-1}$}}
\put(41,4){\makebox(0,0)[l]{$\Gamma_r$}}\put(49,4){\makebox(0,0)[l]{$C_j^*$}}
\put(-10,-4){\makebox(0,0)[l]{$\gamma_0=n_i$}}\put(6.2,-4){\makebox(0,0)[l]{$\gamma_1$}}\put(13,-4)
{\makebox(0,0)[l]{$\gamma_2$}}
\put(31,-4){\makebox(0,0)[l]{$\gamma_{r-1}$}}\put(40,-4){\makebox(0,0)[l]{$\gamma_{r}$}}\put(47,-4)
{\makebox(0,0)[l]{$n_j=\gamma_{r+1}$}}
\end{picture}\\[15pt]

By 1) of Lemma \ref{gamma}, if $i=1,\cdots,r$, then $\gamma_i$
divides $\gamma_{i-1}+\gamma_{i+1}$, we have
$$
s(\gamma_{i-1}, \gamma_i)+s(\gamma_{i+1},
 \gamma_i)=0, \hskip0.3cm \textrm{ for
} i=1,\cdots,r.$$

By 2) of Lemma \ref{gamma}, we have
 \begin{align*}
 s(\gamma_1,\gamma_0)=-s(\gamma_{r+1}, \gamma_{0}), \hskip0.5cm &
 s(\gamma_r,\gamma_{r+1})=-s(\gamma_{0}, \gamma_{r+1}).
 \end{align*}
Hence
 \begin{align*}
{ \sum_{i=1}^{r+1} \chi(\gamma_{i-1},
\gamma_i)\Gamma_{i-1}\Gamma_{i}}
& =\sum_{i=1}^{r+1} (s(\gamma_{i-1}, \gamma_i)+s(\gamma_i, \gamma_{i-1}))\\
& =s(\gamma_{1}, \gamma_0)+s(\gamma_{r}, \gamma_{r+1})+\sum_{i=1}^r (s(\gamma_{i-1}, \gamma_i)+s(\gamma_{i+1}, \gamma_i))\\
 & =-s(\gamma_{r+1},\gamma_0)-s(\gamma_0,\gamma_{r+1})=-\chi(n_i,n_j).
 \end{align*}
 Thus
 \begin{align*}
\mu_{\bar{F^*}}-\beta_{\bar{F^*}}+{\bar{F^*}_{\rm{red}}}^2=-(\mu_{\bar
F}-\beta_{\bar F}+{\bar F_{\rm{red}}}^2).
 \end{align*}
By Lemma \ref{pa},
$p_a(\bar{F^*}_{\rm{red}})=p_a({\bar{F}}_{\rm{red}})$, so
$N_{\bar{F^*}}=N_{{\bar{F}}}$. We get
$\chi_F+\chi_{F^*}=N_{{\bar{F}}}$\ .
\end{proof}

\subsection{Upper and lower bounds on $\chi$}

\begin{theorem}
 $
\frac16{N_{{\bar{F}}}}\leq \chi_F\leq\frac56{N_{{\bar{F}}}}.$  If
$F$ is not semistable, then $\frac{1}{6}\leq \chi_F\leq
\frac{5g}{6}$.
\end{theorem}
\begin{proof}
 By adjunction formula, $2N_F=K_X(F-F_{\rm red})-F_{\rm red}^2$.
By the resolution of the singularities of $F$, we have $p_a(F_{\rm
red})=p_a(\bar F_{\rm red})-\sum_i\frac12(m_i-1)(m_i-2)$, so
$2N_F=2N_{\bar F}-\sum_i(m_i-1)(m_i-2),$ where $m_i\geq 2$ are the
multiplicities of singularities  occurring in the partial
resolutions of $F$. By definition, $\alpha_F=\sum_i(m_i-2)^2$. From
formulas \eqref{1.1},
 \begin{align*}
12\chi_F&=6N_F +F^2_{\rm{red}}+\alpha_F+\mu_F-\beta_F
\\ & =2N_F+ (2N_F+F_{\rm red}^2)+(\mu_F-\alpha_F-\beta_F)+(2N_F+\alpha_F)\\
&=
2N_F+(F-F_{\rm{red}})K_X+(\mu_F-\alpha_F-\beta_F)+2N_{{\bar{F}}}+\sum_{i}
(m_i-2)(m_i-3),
 \end{align*}
 Since $F$ is minimal, $(F-F_{\rm{red}})K_X\geq
 0$. $\mu_F-\alpha_F-\beta_F\geq 0$ is proved in Lemma \ref{mualphabeta}.
Hence $12\chi_F\geq 2N_{{\bar{F}}}$.

 Similarly, $12\chi_{F^*}\geq
2N_{\bar {F^*}}=2N_{{\bar{F}}}$. On the other hand,
$\chi_F+\chi_{F^*}=N_{{\bar{F}}}$, so $12\chi_F\leq
10N_{{\bar{F}}}$.
\end{proof}
\begin{corollary}
 $\chi_F=\frac16{N_{{\bar{F}}}}$ (resp. $\chi_F=\frac56{N_{{\bar{F}}}}$) if and only if $F$ (resp. $F^*$) is a
reduced curve whose singularities
 are at worst ordinary cusps or
nodes.
\end{corollary}
\begin{proof} It follows from Lemma \ref{mualphabeta}.
\end{proof}

\subsection{Applications}
\begin{theorem}
{\rm 1)} If  $f$ is non-trivial, then $ \chi_f\leq
\frac{g}{2}\left(2b-2+\frac{8}{3}s\right). $

  {\rm 2)} If $f$ is isotrivial, then
$ \chi_f\leq \frac{5gs}{6}. $
\end{theorem}
\begin{proof}
1) \  We assume first that $f$ is non-isotrivial.
 Let $F_1$, $\cdots$, $F_s$ be all of the singular fibers.  There exists some
  semistable reduction $\pi:\tilde{C}\to C$ such that

(i) $\pi$ is ramified uniformly over the $s$ critical points of $f$,
and the ramification index of $\pi$ at any ramified point is exactly
$e$.

(ii) $e$ is divisible by $M_{F_i}$ for all $i$, and it can be
arbitrarily large.

 In fact, if $b=g(C)>0$, the existence follows
from  Kodaira-Parshin's construction; if $b=0$, then $s\geq 3$. Thus
one can construct a base change totally ramified over the $s$
points. The existence is induced to the case  $b>0$.

Let $\tilde{f}:\tilde{S}\to \tilde{C}$ be the semistable model and $\tilde{s}$ be the number
of singular fibers of $\tilde{f}$.
Let $\tilde{b}=g(\tilde{C})$ and $d=\deg \pi$. One has
$$2\tilde{b}-2=d(2b-2)+d\left(1-\frac{1}{e}\right)s,\qquad \tilde{s}\leq \frac{ds}{e}.$$
Hence we have
 \begin{eqnarray*}
{ \chi_{f}-\frac{g}{2}\left(2b-2+\frac{8}{3}s\right)}
=\frac{1}{d}\left(\chi_{\tilde{f}}-\frac{g}{2}(2\tilde{b}-2+\tilde{s})\right)
+\frac{g}{2d}\left(\tilde{s}-\frac{ds}{e}\right)+\sum\limits_{i=1}^{s}\left(\chi_{F_i}-\frac{5g}{6}\right).
\end{eqnarray*}
$\chi_{\tilde{f}}\leq \frac{g}{2}(2\tilde{b}-2+\tilde{s})$ is the
Arakelov inequality, so one gets the inequality (1).

2) It is obvious. If $f$ is also non-trivial, then 3) of Theorem
\ref{THMnew1.2}
 implies 2).
\end{proof}

\section{Proof of Theorem \ref{THMnew1.3}}
\subsection{Fibers with high $c_1^2$}
We try to prove Theorem \ref{THMnew1.3}, which implies Theorem
\ref{THMnew1.2}, 2). To describe a fiber $F$, we usually consider
the dual graph of its normal crossing model $\bar F$. We use $\circ$
to denote a $(-2)$-curve, and $\bullet$ a smooth rational curve but
not a $(-2)$-curve. The number beside is the multiplicity of the
curve in $\bar F$. The self-intersection number of each component
$\bullet$ can be determined by using Zariski's lemma.

The following fiber $F$ of genus $g$ satisfies
$c_1^2(F)=4g-\frac{11}{2}$, $c_2(F)=2g+\frac{5}{2}$,
$\chi_F=\frac{g}{2}-\frac{1}{4}$.
\begin{example}\label{LiZi}  $F=(g-1)F_0$, where
$F_0$ is curve of genus $2$ whose
 dual graph   is as follows.
 \setlength{\unitlength}{0.9mm}

\hfill\begin{picture}(168,15)(-60,-5)
\put(-35,0.5){\makebox(0,0)[l]{}}
\put(0,0){\circle*{1.5}}\put(0.7,0.7){\line(1,1){5}}\put(0.7,-0.7){\line(1,-1){5}}
\put(6.5,6.5){\circle{1.5}}
\put(6.5,-6.5){\circle{1.5}}\put(13,0){\circle{1.5}}\put(12.3
,0.7){\line(-1,1){5}} \put(12.3,-0.7){\line(-1,-1){5}} \put(14
,0){\line(1,0){5}}\put(20,0){\circle{1.5}}
\put(-3,0){\makebox(0,0)[l]{$2$}}\put(3.5,7){\makebox(0,0)[l]{$3$}}\put(3.5,-7)
{\makebox(0,0)[l]{$3$}}
\put(13.5,3){\makebox(0,0)[l]{$4$}}\put(20.0,3){\makebox(0,0)[l]{$2$}}
\end{picture}\\[-4pt]
\end{example}

\begin{lemma}\label{Artin} {\bf (Artin \cite{Artin})} \ Let $D$ be an effective divisor on a surface.
Suppose $D^2<0$ and $D\Gamma_i\leq 0$ for any component $\Gamma_i$
of $D$. Then $D$ is a negative curve, i.e., the intersection matrix
$(\Gamma_i\Gamma_j)$ is negative definite.
\end{lemma}

In  what follows, we always assume that $F$ satisfies $c_1^2(F)>
4g-\frac{11}2$,
 namely,
\begin{align}\label{c1bound}
4p_a({\bar{F}}_{\rm{red}})-F_{\rm{red}}^2+\beta_F^{-}+\sum\limits_{i=1}^rm_i(m_i-2)<\frac{11}{2}.
\end{align}
Note that each term on the left hand side of \eqref{c1bound} is
non-negative.
\begin{lemma}\label{7A} \ {\rm 1)} \  $m_i\leq 3$ for all $i$ and at most one $m_i$ is
equal to $3$.  So $F_{\rm red}$ admits at most one singular point
$p$ which is not a node. In fact, $p$ is of types $A_2$, $A_3$ or
$D_4$.

{\rm 2)} \ $\bar F_{\rm red} ^2\leq -1$.

{\rm 3)} \ $ p_a(\bar F_{\rm red})=0$, so $\bar F$ is a tree of
smooth rational curves.

{\rm 4)} \ $p_a(F_{\rm red})\leq 1$, with equality iff one singular
point $p$ of $F_{\rm red}$  is not a node as in {\rm 1)}.
\end{lemma}
\begin{proof} 1) follows from the inequality $\sum\limits_{i=1}^rm_i(m_i-2)<11/2$ and Lemma \ref{ADE5}.

2) \ \eqref{c1bound} implies that  $p_a({\bar{F}}_{\rm{red}})\leq
1$, i.e., $K\bar F_{\rm red} + \bar F_{\rm red}^2\leq 0$.
 If $\bar F_{\rm red}^2=0$, then by Zariski's lemma, $\bar F=n\bar F_{\rm red}$ for
  some positive integer $n$. Since $K\bar F_{\rm red}\leq 0$, we see that
  $2g-2=K\bar F=nK\bar F_{\rm red}\leq 0$, a contradiction. So $\bar F_{\rm red} ^2\leq -1$.

3) \ Note that  $p_a({\bar{F}}_{\rm{red}})\leq 1$.
 Suppose that $p_a({\bar{F}}_{\rm{red}})=1$. Then $\sum\limits_{i=1}^rm_i(m_i-2)\leq 3/2$,
 so all $m_i=2$ and
  $F_{\rm{red}}=\bar F_{\rm{red}}$ is a nodal curve.  We see also that $-F_{\rm red} ^2<3/2$, so $
  F_{\rm{red}}^2=-1$, $KF_{\rm red}=1$, and $F$ consists of one
  $(-3)$-curve and some $(-2)$-curves. Now from \eqref{c1bound}, we get $\beta_{F}^{-}<\frac{1}{2}$.

If one $(-2)$-curve $E$ in $F$ meets at only one point with the
other components, then $E$ is the end point of some H-J branch, and
the contribution of $E$ to $\beta_F^-$ is at least $\frac12$, a
contradiction. Hence any $(-2)$-curve is a point in some loops in
the dual graph of $F$. Because $p_a(F_{\rm red})=1$, there is only
one loop in the dual graph. Hence the dual graph of $F$ consists of
one loop. Now we see that $F_{\rm{red}}\Gamma\leq 0$  for each
irreducible component $\Gamma$. Combine with
 $F_{\rm{red}}^2<0$, we know that $F$ is a negative curve (Lemma \ref{Artin}), a contradiction.

4) \  By Lemma \ref{7A} and \eqref{ArithGenus}, we have $p_a(F_{\rm
 red})=p_a(\bar F_{\rm
 red})+\sum_i\frac12(m_i-1)(m_i-2)\leq 1$.
\end{proof}

\subsection{The case $p_a(F_{\rm red})=1$}

\begin{proposition} If  $F$ is not a nodal curve, then $F_{\rm red}$ has one singular point of
 type $A_3$.  The normal crossing model of $F$  is of type 21. 
\end{proposition}
\begin{proof}
 In this case,  $F$ has a unique
singularity $p$ of types $A_2$,  $A_3$, or $D_4$.
$p_a(F_{\rm{red}})=1$, one has
$-F_{\rm{red}}^2+\beta_F^{-}<\frac{5}{2}$. Since
$p_a({\bar{F}}_{\rm{red}})=0$, the dual graph of ${\bar{F}}$ is a
tree of rational curves.

\smallskip {\it Case ${A_2}$}: \ Suppose that  $p$ is of type $A_2$.
Then the contribution of $p$ to $\beta_F^{-}\geq \frac{5}{6}$,
  so
$-F_{\rm{red}}^2< \frac{5}{3}$. We have
$-F_{\rm{red}}^2=F_{\rm{red}}K_X=1$, and $\beta_F^{-}<\frac{3}{2}$.
Let $C_1$ be the irreducible component passing through $p$. Then
$KC_1=1$ and $F_{\rm{red}}-C_1$ is composed  of some $ADE$ curves.
Suppose that $F_{\rm{red}}-C_1$ contains at least two $(-2)$-curves
as the end points in the dual graph of $F$. Then their contributions
to
 $\beta_F^{-}$ is at least $1$. So $\beta_F^{-}\geq
 1+\frac{5}{6}>\frac{3}{2}$, a contradiction. So only one $(-2)$-curve is an end point.
  On the other hand, from $p_a(\bar F_{\rm red})=0$,
 we see that $F$ contains no loop. Hence  $F$ is a H-J chain with an end point
  $C_1$. It implies $F$ is a negative  curve, a contradiction.

\smallskip {\it Case $A_3$}: \  Assume that $p$ is of type $A_3$. The
contribution of $p$  to $\beta_F^{-}\geq \frac{1}{2}$ and so
$-F_{\rm{red}}^2<2$. Now we have $-F_{\rm{red}}^2=F_{\rm{red}}K_X=1$
and $\beta_F^-<\frac32$. $F$ consists of some $(-2)$-curves and one
curve $C_1$ passing through $p$. Note that $\bar F$ is a tree of
rational curves, so no node is a singular point of $C_1$, namely
$C_1$ is smooth except at $p$.  If $C_1$ is singular at $p$, then
there is no $(-2)$-curve passing through $p$. Similar to the
discussion above, only one $(-2)$-curve is the end point. Now we
know that $F$ is a chain of $(-2)$-curves and $C_1$, so $F$ is a
negative curve, a contradiction. Hence $C_1$ is smooth at $p$ and
there is a $(-2)$-curve $C_2$ tangent to $C_1$ at $p$. Because
$KC_1=1$, $C_1$ is a $(-3)$-curve.

 There is  a $(-2)$-curve $C_2$ tangent $C_1$ at $p$.  $F_{\rm{red}}-C_1-C_2$ consists of $ADE$ curves.
Because only one $(-2)$-curve is the end point, we know that
$\Gamma=F_{\rm{red}}-C_1-C_2$ is just a curve of type $A_n$.

   If $C_1$ intersects $\Gamma$, then $F_{\rm{red}}=\Gamma+C_1+C_2$ is a chain. One can
   prove that $F$ is a negative curve by Lemma \ref{Artin}, a
   contradiction. So $C_1$ is disjoint with $\Gamma$. $C_2+\Gamma$ is
   a connected curve of type $A_{n+1}$.

By using Zariski's lemma, one can determine the multiplicities of
all irreducible components in $F$ and the number of $(-2)$-curves.
Finally, we get the fiber of type 21.

\smallskip {\it Case $D_4$}: \ Suppose that  $p$ is of type $D_4$.
Because $\bar F$ is a tree of rational curves, the three local
branches of $F$ at $p$ come from 3 different components $C_1$, $C_2$
and $C_3$ of $F$. At least one component, say $C_1$, is not a
$(-2)$-curve since $g\geq 2$.
 Suppose that  $C_2$ is not a $(-2)$-curve. Then
  $ F_{\rm{red}}K_X\geq 2$.  Recall that $F_{\rm{red}}K_X=-F_{\rm{red}}^2\leq 2$,
 one has  $F_{\rm{red}}K_X=-F_{\rm{red}}^2=2$.
  Thus $\beta_F^{-}<\frac{1}{2}$ and $C_1K_X=C_2K_X=1$, namely, $C_1$ and $C_2$ are $(-3)$-curves.
  Hence
  $F_{\rm{red}}-C_1-C_2$ consists of $ADE$-curves whose contributions to $\beta_F^{-}\geq \frac{1}{2}$,
   a contradiction. Therefore $C_2$ and $C_3$ must be $(-2)$-curves. Similarly, we can prove that $C_1$
   is  not a $(-4)$-curve, hence it is a $(-3)$-curve. One can prove
   also that the other curves in $F$ are  $(-2)$-curves.
   Now we have
   $-F^2_{\rm red}=KF_{\rm red}=1$,  and $\beta_F^{-}<\frac{3}{2}$.

  The normal crossing model $\bar F$ of $F$ is obtained by blowing
  up $F$ at $p$. Since the intersection matrix of $C_1$, $C_2$ and
  $C_3$ is negative definite, $\Gamma=F_{\rm red}-C_1-C_2-C_3$
  consists of $s\geq 1$ connected $ADE$-curves
  $\Gamma_1,\cdots,\Gamma_s$. From $\beta_F^{-}<\frac{3}{2}$, we see
  that at most two end points are $(-2)$-curves, so $s\leq 2$.  Let $r_i-1$
  be the number of irreducible components  of $\Gamma_i$. Since $\beta_F<\frac{3}{2}$,
 $s\leq 2$.

Suppose $s=2$. Since at most two end points are $(-2)$-curves,
$\Gamma_1$ and $\Gamma_2$ are of types $A_{r_1-1}$ and $A_{r_2-1}$
respectively. In $\bar F$, $C_1^2=-4$, $C_2^2=C_3^2=-3$. $\Gamma_i$
meets $C_j$ at one point, so we obtain a H-J branch of type
$[2,2,\cdots,2,e_{r_i}]$, where $e_{r_i}=-C_j^2$.

Symmetrically, we only need to consider two cases: I) \ $\Gamma_1$
meets $C_2$ and $\Gamma_2$ meets $C_1$; II) \
$\Gamma_1$ meets $C_2$
and $\Gamma_2$ meets $C_3$.

In case I), from Zariski's lemma, one can find an equality
 $
\frac23=\frac{r_1}{2r_1+1}+\frac{r_2}{3r_2+1},
 $
i.e., $ 1=\frac3{2r_1+1}+\frac2{3r_2+1}. $ We claim that there are
no nonnegative integers $r_1$ and $r_2$ satisfying this equation.
Indeed, for $r_1=0$, $1$ or $2$, this equation has no nonnegative
integral solution $r_2$. So we can assume that $r_1\geq 3$.
Similarly, we can assume also that $r_2\geq 2$. Now the right hand
side is less than $1$. So case I) does not occur.

In case II), we have similarly
$\frac34=\frac{r_1}{2r_1+1}+\frac{r_2}{2r_2+1}$, i.e.,
$\frac12=\frac1{2r_1+1}+\frac1{2r_2+1}$. It is obvious that this
equation has no integral solutions. So case II) can not occur.

Suppose $s=1$.  If $\Gamma_1$ is of type $A_{r_1-1}$, by Zariski's
Lemma, we have either $\frac{12}{5}=\frac{2r_1+1}{r_1}$ or
$3=3+\frac1{r_1}$. These equations have no integral solutions. So
this case dose not occur.

Finally, we assume that $\Gamma_1$ is not of type $A_{r_1-1}$. Now
we see that there are two end points which are $(-2)$-curves, so the
contribution of them to $\beta^-_F$ is at least 1. On the other
hand, the contribution of the two components disjoint from
$\Gamma_1$ are at least $\frac14+\frac13=\frac7{12}$. So
$\beta_F^-\geq 1+\frac7{12}>\frac32$, a contradiction.

Up to now, we have proved that the case $D_4$ does not occur.
\end{proof}

\subsection{The case $p_a(F_{\rm red})=0$}
 From now on, we always assume that $F_{\rm{red}}$ is a tree of smooth rational curves, namely,
 $p_a(F_{\rm{red}})=0$. Hence \eqref{c1bound} becomes $-F_{\rm{red}}^2+\beta_F^{-}< \frac{11}{2}$.
Namely,
 \begin{equation}\label{7B}
 F_{\rm{red}}K_X+\beta_F^{-}< \frac{7}{2}.
 \end{equation}

\begin{lemma} We have $F_{\rm{red}}K_X=1$ and $F_{\rm{red}}^2=-3$. Namely, $F_{\rm{red}}$
consists of a $(-3)$-curve  and some $(-2)$-curves. So
$\beta^-_F<\frac52$.
\end{lemma}
\begin{proof}
Suppose that  $F_{\rm{red}}K_X\geq 2$. Let $s$ be the number of
$(-2)$-curves as the end points in the dual graph of $F$.
$\beta_F^{-}<\frac{3}{2}$ implies $s\leq2$. Assume that the dual
graph of $F$ contains
  $r$ end points. Obviously $r\geq 3$.

   We claim first that $r=3$, $s=1$ and $F_{\rm{red}}K_X=2$.

 Indeed, there are at least $r-s$ end points which are not $(-2)$-curves.
So $F_{\rm{red}}K_X\geq r-s$ and $\beta_F^{-}< \frac{7}{2}+s-r$. On
the other hand, $\beta_F^{-}>\frac{s}{2}$. So $s\geq 2r-6$. Note
that $s\leq 2$, we get $r\leq 4$. If $r= 4$, then $s=2$. Then we see
that $1<\beta_F^{-}<\frac{3}{2}$ and $F_{\rm{red}}K_X=2$. It implies
also that two of the end points are $(-3)$-curves. Thus
 $\beta_F^{-}\geq 2(\frac{1}{2}+\frac{1}{3})>\frac{3}{2}$, a contradiction.  So $r=3$.

If $F_{\rm{red}}K_X=3$, then $\beta_F^{-}<\frac{1}{2}$. So any end
point is a $(-3)$-curve. Thus  $\beta_F^{-}\geq 1$ , a
contradiction. Hence $F_{\rm{red}}K_X=2$. It implies $s\geq
r-F_{\rm{red}}K_X= 1$.

Suppose that $s=2$. Since $r=3$ and $F$ is a tree of rational
curves, $F$ has two H-J chains of type $A_n$ and one H-J chain whose
end point is a $(-e)$-curve, $e=3$ or $4$. We have seen in
\S\ref{Sect3.2} that the multiplicities in a H-J branch increase
strictly from the end point to the other side.

Suppose $e=4$. From $F_{\rm{red}}K_X=2$, we see that all other
components are $(-2)$-curves. The dual graph of $F$ is as follows.
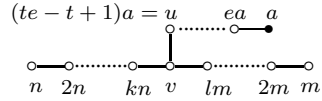
\setlength{\unitlength}{0.7mm}

\hfill\begin{picture}(168,15)(-30,-1)\footnotesize
\put(-35,-0.5){\makebox(0,0)[l]{}}
 \multiput(19,0)(7,0){3}{\circle{1.5}}\multiput(20,0)(7,0){2}{\line(1,0){5}}
 \multiput(0,0)(7,0){2}{\circle{1.5}}
 \multiput(1,0)(7,0){1}{\line(1,0){5}}
 \multiput(8.5,0)(1,0){10}{\circle*{0.5}}
 \multiput(39,7)(0,7){1}{\line(1,0){5}}
 \put(26,1){\line(0,1){5}}\put(26,7){\circle{1.5}}
 \multiput(28,7)(1,0){9}{\circle*{0.5}}
  \multiput(44,9)(1,0){1}{$a$}
  \multiput(36,9)(1,0){1}{$ea$}
   \multiput(-4,9)(1,0){1}{$(te-t+1)a=u$}
 \multiput(38,7)(0,7){1}{\circle{1.5}}
 \put(44.5,7){\circle*{1.5}}
 \multiput(34.5,0)(1,0){10}{\circle*{0.5}}\multiput(45,0)(7,0){2}{\circle{1.5}}
 \multiput(46,0)(7,0){1}{\line(1,0){5}}
\put(-0.5,-4){\makebox(0,0)[l]{$n$}}\put(5.5,-4){\makebox(0,0)[l]{$2n$}}
\put(17.5,-4){\makebox(0,0)[l]{$kn$}}
\put(32.5,-4){\makebox(0,0)[l]{$lm$}}\put(42.5,-4){\makebox(0,0)[l]{$2m$}}
\put(50.5,-4){\makebox(0,0)[l]{$m$}}\put(25,-4){\makebox(0,0)[l]{$v$}}
\end{picture}\\[15pt]
where $(k+1)n=(l+1)m=((t+1)e-t)a=v$ ($1\leq k\leq l$) and
${kn+lm+u}=2v$ by Zariski's lemma, so we have
$\frac{k}{k+1}+\frac{l}{l+1}+\frac{u}{v}=2$. It is easy to see that
$$\frac{k}{k+1}+\frac{l}{l+1}+\frac{1}{4}\leq\beta_F^{-}<\frac{3}{2}.$$
So either $k=l=1$, or $k=1$ and $l=2$. Now we see that
$\frac{u}{v}=1$ or $\frac76$. On the other hand, $v>u$, a
contradiction.

If $e=3$, then  there exists another $(-3)$-curve $E$. In fact, $E$
can not be in the center, otherwise $3v=kn+lm+u<v+v+v$, a
contradiction. $E$ can not be in the vertical branch, otherwise, we
have
$$\frac{k}{k+1}+\frac{l}{l+1}+\frac{1}{3}\leq\beta_F^{-}<\frac{3}{2},$$
it implies $k=l=1$, i.e., $n=m$ and $v=2n$. Since ${kn+lm+u}=2v$, we
have $u=v$, a contradiction with $v>u$. Hence $E$ must be a
component of the horizontal branch. Without loss of generality, we
assume that $E$ is on the right branch. Consider the contribution to
$\beta^-_F$, we have $k=1$ and $E$ intersects with the $(-2)$-curve
at the end. The dual graph of $F$ is as follows.
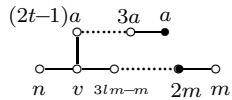
\setlength{\unitlength}{0.7mm}

\hfill\begin{picture}(168,18)(-30,-1)\footnotesize
\put(-35,-0.5){\makebox(0,0)[l]{}}
 \put(45.2,0){\circle*{1.5}}
 \put(52,0){\circle{1.5}}
 \multiput(19,0)(7,0){3}{\circle{1.5}}\multiput(20,0)(7,0){2}{\line(1,0){5}}
 \multiput(36.8,7)(0,7){1}{\line(1,0){5}}
 \put(26,1){\line(0,1){5}}\put(26,7){\circle{1.5}}
 \multiput(27,7)(1,0){9}{\circle*{0.5}}
 \multiput(36,7)(0,7){1}{\circle{1.5}}
 \put(42.5,7){\circle*{1.5}}
 \multiput(34.5,0)(1,0){10}{\circle*{0.5}}
 \multiput(45,0)(7,0){1}{\circle{1.5}}
 \multiput(46,0)(7,0){1}{\line(1,0){5}}
\put(17.5,-4){\makebox(0,0)[l]{$n$}}
\put(29,-4){\makebox(0,0)[l]{\tiny $3lm\hskip-0.1cm-\hskip-0.1cm
m$}}\put(43.5,-4){\makebox(0,0)[l]{$2m$}}
\put(51,-4){\makebox(0,0)[l]{$m$}}
\put(25,-4){\makebox(0,0)[l]{$v$}}
\put(13,10){\makebox(0,0)[l]{$(2t\hskip-0.09cm-\hskip-0.1cm1)a$}}
\put(40,10){\makebox(0,0)[l]{  $a$}}
\put(33.5,10){\makebox(0,0)[l]{$3a$}}
\end{picture}\\[15pt]
We have $v=2n=(2t+1)a=(3l+2)m$, and $n+(2t-1)a+(3l-1)m=2v$. It
implies
$$
\frac12+\frac{2t-1}{2t+1}+\frac{3l-1}{3l+2}=2\, ,
$$
i.e.,
$$
\frac{2}{2t+1}+\frac{3}{3l+2}=\frac12\, .
$$
This equation has only one solution $t=3$ and $l=4$. Now we can
compute $\beta^-_F=\frac{23}{14}>\frac32$, a contradiction.

 We have proved that $s=2$ can not occur. So $s=1$. The claim is
 proved.

 Finally, we need to exclude the case in the claim.

 $F$ has exactly two H-J branches whose end points are
   $(-3)$-curves. The remaining H-J branch is of type $A_n$ which contains $k$
   vertexes. The dual graph
   is as follows.
   \setlength{\unitlength}{0.7mm}

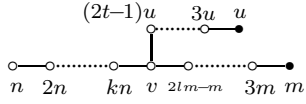
\hfill\begin{picture}(168,18)(-30,-1)\footnotesize
\put(-35,-0.5){\makebox(0,0)[l]{}}
 \put(52,0){\circle*{1.5}}
 \multiput(19,0)(7,0){3}{\circle{1.5}}\multiput(20,0)(7,0){2}{\line(1,0){5}}
 \multiput(0,0)(7,0){2}{\circle{1.5}}
 \multiput(1,0)(7,0){1}{\line(1,0){5}}
 \multiput(8.5,0)(1,0){10}{\circle*{0.5}}
 \multiput(36.8,7)(0,7){1}{\line(1,0){5}}
 \put(26,1){\line(0,1){5}}\put(26,7){\circle{1.5}}
 \multiput(27,7)(1,0){9}{\circle*{0.5}}
 \multiput(36,7)(0,7){1}{\circle{1.5}}
 \put(42.5,7){\circle*{1.5}}
 \multiput(34.5,0)(1,0){10}{\circle*{0.5}}
 \multiput(45,0)(7,0){1}{\circle{1.5}}
 \multiput(46,0)(7,0){1}{\line(1,0){5}}
\put(-0.5,-4){\makebox(0,0)[l]{$n$}}\put(5.5,-4){\makebox(0,0)[l]{$2n$}}
\put(17.5,-4){\makebox(0,0)[l]{$kn$}}
\put(29,-4){\makebox(0,0)[l]{\tiny $2lm\hskip-0.1cm-\hskip-0.1cm
m$}}\put(43.5,-4){\makebox(0,0)[l]{$3m$}}
\put(51,-4){\makebox(0,0)[l]{$m$}}
\put(25,-4){\makebox(0,0)[l]{$v$}}
\put(13,10){\makebox(0,0)[l]{$(2t\hskip-0.09cm-\hskip-0.1cm1)u$}}
\put(40,10){\makebox(0,0)[l]{  $u$}}
\put(33.5,10){\makebox(0,0)[l]{$3u$}}
\end{picture}\\[15pt]
where $(2l+1)m=(2t+1)u=(k+1)n=v$ ($l\leq t$) and $\frac{2l-1}{2l+1}+
\frac{2t-1}{2t+1}+\frac{k}{k+1}=2$ by Zariski's lemma, i.e.,
$\frac1{2l+1}+\frac1{2t+1}=\frac{k}{2k+2}$. Then we have
\begin{align*}
\beta_F^{-}&=\frac{l}{2l+1}+\frac{t}{2t+1}+\frac{k}{k+1} \\
&=1-\dfrac12\left(\frac1{2l+1}+\frac1{2t+1}\right)+\dfrac{k}{k+1}\\
&=1-\dfrac14\cdot\dfrac{k}{k+1}+\dfrac{k}{k+1}
=1+\dfrac34\cdot\dfrac{k}{k+1}<\dfrac32,
\end{align*}
we get $k=1$. It is easy to see that the equation
$\frac1{2l+1}+\frac1{2t+1}=\frac14$ has no positive integral
solutions $l$ and $t$. So the case in the claim is excluded.
 Hence
$F_{\rm{red}}K_X=2$ is impossible.

The lemma is finally proved.
 \end{proof}

Now $F$ consists of one $(-3)$-curve $C_0$ and some connected $ADE$
curves $\Gamma_1,\cdots, \Gamma_r$. Let $Z_i$ be the fundamental
cycle supported on $\Gamma_i$. Then $Z_i^2=-2$. See (\cite{BPV},
Ch.III, \S3) for the list of $Z_i$.

 Since $(C_0+Z_i)^2\leq 0$, $1\leq C_0Z_i\leq 2$. If $C_0Z_i=2$,
 then $Z_i$ can not be of type $A_n$, otherwise $Z_i$ is
 reduced and $C_0Z_i=2$ implies that $F$ is not a tree. Hence $Z_i$ must
 be of types $E_k$ or $D_n$.

\begin{lemma} If  $C_0Z_i=2$ for some $i$, then $g=2$ and $F$ is
 of types 10 $\sim$ 16.
 \end{lemma}
\begin{proof}
 {\it Step 1}: \  There is at most one $Z_i$ such that  $C_0Z_i=2$.
Otherwise if $Z_i$ and $Z_j$ satisfy $C_0Z_i=C_0Z_j=2$, then
   $(C_0+Z_1+Z_2)^2=1$, a contradiction. Without loss of generality,
   we assume $C_0Z_1=2$ and $C_0Z_i=1$ for all
   $i\geq 2$.

\smallskip
{\it Step 2}: \ Suppose $r\geq 3$.  One can check that
$$(2C_0+2Z_1+Z_2+Z_3)^2=0.$$
Note that $F$ is simply connected, $F$ can not be a multiple fiber
(\cite{Xi90}, p.389). So $F=2C_0+2Z_1+Z_2+Z_3$. Let $C_2$ be an
irreducible component of $Z_2$ such that $C_2Z_2<0$. From $FC_2=0$
we get $C_2Z_2=-2C_0C_2$. Since $(Z_2-C_2)^2\leq 0$, we have
$C_2Z_2\geq -2$, so $C_2Z_2=-2$ and $(C_2-Z_2)^2=0$, i.e., $Z_2=C_2$
is just one $(-2)$-curve. Similarly, $Z_3$ is also a $(-2)$-curve.
Recall that supp$(Z_1)$ is a curve of types $D_n$ or $E_k$, and
$C_0$ meets with $Z_1$ at the component $E$ with $EZ_1<0$. Because
$\beta_F^{-}<\frac{5}{2}$, $Z_1$ can not be of type $D_n$. Now one
can check that the
 possibilities are just the fibers of types 11, 12 and 13.

\smallskip
{\it Step 3}: \ Suppose $r=2$. Let $C_1$ be the irreducible
component of $Z_1$ such that $C_1Z_1<0$. Since  $Z_1$ is not a curve
of type $A_n$, one can check from the list that $C_1Z_1=-1$. Then
$(2C_0+2Z_1+Z_2-C_1)^2=-4C_0C_1$. If $C_0C_1=0$,
$F=2C_0+2Z_1+Z_2-C_1$. By Zariski's lemma, $0=FZ_1=-C_1Z_1$, a
contradiction. So $C_0C_1=1$.

Let $C_2$ be an irreducible component of $Z_2$ such that $C_0C_2=1$.
Since $C_0Z_2=1$, the multiplicity of $C_2$ in $Z_2$ is $1$. If
$Z_2C_2<0$, then one can check that $Z_2$ is of type $A_n$, $C_2$ is
at the end of $Z_2$ and $C_2Z_2=-1$. Consider $D=C_0+Z_1+Z_2$,  one
can check that $D\Gamma\le 0$ for each irreducible $\Gamma$ of $D$,
e.g., $C_1D=0$ and $C_2D= 0$. $D^2=-1$. By Lemma \ref{Artin}, $D$ is
a negative curve, a contradiction. Hence $Z_2C_2=0$.

Now we have
$$(2C_0+2Z_1+Z_2+C_2)^2=0.$$ Thus $F=2C_0+2Z_1+Z_2+C_2$.
Since $Z_2C_2=0$, $Z_2$ can not be irreducible. There is another
component $C_3$ of $Z_2$ such that $Z_2C_3<0$. Since
$0=FC_3=Z_2C_3+C_2C_3$, we see that $Z_2C_3=-1$, and $C_2C_3=1$.
Check  each type of $ADE$ fundamental cycles, one find that $Z_2$
must be of type $D_n$. From $\beta^-_F<\frac52$, we see that the
dual graph of $F$ has at most 4 $(-2)$-curves as its end points, so
$Z_1$ can not be of type $D_n$. Now we obtain that $F$ is of types
14, 15 and 16.

\smallskip
{\it Step 4}: \ Suppose $r=1$.  Let $C_1$ be the irreducible
component of $Z_1$ such that $C_0C_1=1$.  If $C_1Z_1<0$, $C_0+Z_1$
is a negative cycle by Lemma \ref{Artin}, a contradiction. So
$C_1Z_1=0$. Let $C_2$ be another irreducible component of $Z_1$ such
that $Z_1C_2<0$, one cane check that $Z_1C_2=-1$.

If $C_1C_2=0$, then $$(2C_0+2Z_1+C_1-C_2)^2=0.$$ So
$F=2C_0+2Z_1+C_1-C_2$. Thus $0=FC_0=-1-C_0C_2$, a contradiction.
Hence $C_1C_2=1$. By checking each type of $ADE$ fundamental cycles,
we see that $Z_1$ is of type $D_n$. Note that $C_2$ is unique in
$Z_1$. We claim that $C_1$ is not at the end of $D_n$. Otherwise, by
Lemma \ref{Artin}, $C_0+Z_1+C_1$ is a negative cycle, a
contradiction. So the position of $C_1$ is determined. Now we see
easily that $F$
is just the fiber of type 10. 
\end{proof}

From now on we always assume that $C_0Z_i=1$ for all $i$. So $C_0$
meets with a component whose multiplicity in $Z_i$ is 1. From
Zariski's lemma, one can determine the multiplicities of the
irreducible components of $Z_i$ in $F$ whenever the multiplicity of
$C_0$ is determined. The following are all possible partial dual
graphes of $C_0$ and $Z_i$ in $F$.

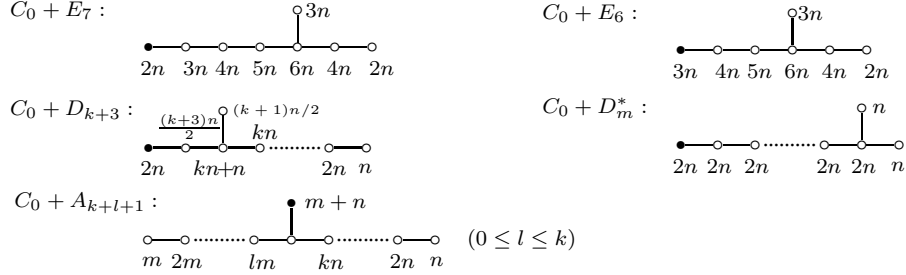
\setlength{\unitlength}{0.7mm}

\hfill\begin{picture}(178,19)(-30,-5)\footnotesize
\put(-35,0.5){\makebox(0,0)[l]{}} \put(-7,0){\circle*{1.5}}
\multiput(-6,0)(7,0){6}{\line(1,0){5}}\multiput(0,0)(7,0){6}{\circle{1.5}}
\multiput(21,7)(7,0){1}{\circle{1.5}}\multiput(21,1)(7,0){1}{\line(0,1){5}}
\put(-8.5,-4){\makebox(0,0)[l]{$2n$}}
\put(-0.5,-4){\makebox(0,0)[l]{$3n$}}
\put(5.5,-4){\makebox(0,0)[l]{$4n$}}
\put(12.5,-4){\makebox(0,0)[l]{$5n$}}\put(19.5,-4){\makebox(0,0)[l]{$6n$}}
\put(22.5,7){\makebox(0,0)[l]{$3n$}}
\put(26.5,-4){\makebox(0,0)[l]{$4n$}}\put(34.5,-4){\makebox(0,0)[l]{$2n$}}
\put(-33,7){\makebox(0,0)[l]{$C_0+E_7:$}}
\end{picture}

\setlength{\unitlength}{0.7mm}

\hfill\begin{picture}(78,12)(-30,-17)\footnotesize
\put(-35,0.5){\makebox(0,0)[l]{}} \put(-7,0){\circle*{1.5}}
\multiput(-6,0)(7,0){5}{\line(1,0){5}}\multiput(0,0)(7,0){5}{\circle{1.5}}
\multiput(14,7)(7,0){1}{\circle{1.5}}\multiput(14,1)(7,0){1}{\line(0,1){5}}
\put(-8.5,-4){\makebox(0,0)[l]{$3n$}}
\put(-0.5,-4){\makebox(0,0)[l]{$4n$}}
\put(5.5,-4){\makebox(0,0)[l]{$5n$}}
\put(12.5,-4){\makebox(0,0)[l]{$6n$}}\put(19.5,-4){\makebox(0,0)[l]{$4n$}}
\put(15,7){\makebox(0,0)[l]{$3n$}}
\put(27.5,-4){\makebox(0,0)[l]{$2n$}}
\put(-33,7){\makebox(0,0)[l]{$C_0+E_6:$}}
\end{picture}

\setlength{\unitlength}{0.7mm}

\hfill\begin{picture}(178,12)(-30,-11)\footnotesize
\put(-35,0.5){\makebox(0,0)[l]{}} \put(-7,0){\circle*{1.5}}
\multiput(-6,0)(7,0){3}{\line(1,0){5}}
\multiput(0,0)(7,0){3}{\circle{1.5}}
\multiput(16,0)(1,0){10}{\circle*{0.5}}
\multiput(27,0)(7,0){2}{\circle{1.5}}
\multiput(28,0)(7,0){1}{\line(1,0){5}}
\multiput(7,7)(7,0){1}{\circle{1.5}}
\multiput(7,1)(0,7){1}{\line(0,1){5}}
\put(-8.5,-4){\makebox(0,0)[l]{$2n$}}
\put(-5.5,4){\makebox(0,0)[l]{\tiny $\frac{(k+3)n}{2}$}}
\put(0,-4){\makebox(0,0)[l]{ $kn\hskip-0.1cm+\hskip-0.1cm n$}}
\put(11.0,3){\makebox(0,0)[l]{ $kn$}}
\put(8.8,7){\makebox(0,0)[l]{\tiny${(k+1)n}/{2}$}}
\put(25.5,-4){\makebox(0,0)[l]{$2n$}}
\put(32.5,-4){\makebox(0,0)[l]{$n$}}
\put(-33,7){\makebox(0,0)[l]{$C_0+D_{k+3}:$}}
\end{picture}

\setlength{\unitlength}{0.7mm}

\hfill\begin{picture}(78,8)(-30,-20)\footnotesize
\put(-35,0.5){\makebox(0,0)[l]{}} \put(-7,0){\circle*{1.5}}
\multiput(-6,0)(7,0){2}{\line(1,0){5}}\multiput(0,0)(7,0){2}{\circle{1.5}}
\multiput(9,0)(1,0){10}{\circle*{0.5}}\multiput(20,0)(7,0){3}{\circle{1.5}}
\multiput(21,0)(7,0){2}{\line(1,0){5}}
\multiput(27,7)(7,0){1}{\circle{1.5}}\multiput(27,1)(0,7){1}{\line(0,1){5}}
\put(-8.5,-4){\makebox(0,0)[l]{$2n$}} \put(-3.5,
-4){\makebox(0,0)[l]{ $2n$}} \put(3.5,-4){\makebox(0,0)[l]{ $2n$}}
\put(18.5,-4){\makebox(0,0)[l]{$2n$}}
\put(29,7){\makebox(0,0)[l]{$n$}}
\put(25,-4){\makebox(0,0)[l]{$2n$}}\put(32.5,-4){\makebox(0,0)[l]{$n$}}
\put(-33,7){\makebox(0,0)[l]{$C_0+D_m^*:$}}
\end{picture}
\setlength{\unitlength}{0.7mm}
\hfill\begin{picture}(178,0)(-69,-8)\footnotesize
\put(-35,0.5){\makebox(0,0)[l]{}} \put(-7,0){\circle{1.5}}
\multiput(-6,0)(7,0){2}{\line(1,0){5}}\multiput(0,0)(7,0){2}{\circle{1.5}}
\multiput(9,0)(1,0){10}{\circle*{0.5}}
\multiput(-18,0)(1,0){10}{\circle*{0.5}}
\multiput(20,0)(7,0){2}{\circle{1.5}}
\multiput(21,0)(7,0){1}{\line(1,0){5}}
\multiput(-26,0)(7,0){1}{\line(1,0){5}}\multiput(-27,0)(7,0){2}{\circle{1.5}}
\multiput(0,7)(7,0){1}{\circle*{1.5}}
\multiput(0,1)(0,7){1}{\line(0,1){5}}
\put(-8,-4){\makebox(0,0)[l]{$lm$}}
 \put(3.5,-4){\makebox(0,0)[l]{ $kn$}}
\put(18.5,-4){\makebox(0,0)[l]{$2n$}}
\put(26,-4){\makebox(0,0)[l]{$n$}}
\put(2.5,7){\makebox(0,0)[l]{$m+n$}}
 \put(-24,-4){\makebox(0,0)[l]{ $2m$}}
\put(-28,-4){\makebox(0,0)[l]{$m$}}
\put(-52,7){\makebox(0,0)[l]{$C_0+A_{k+l+1}:$}}
\put(33,0){\makebox(0,0)[l]{$(0\leq l\leq k)$}}
\end{picture}

Recall that $F_{\rm{red}}-C_0$ consists of $r$ connected
  components $\Gamma_1$, $\cdots$, $\Gamma_r$ and
  $\beta_F^{-}<\frac{5}{2}$.
Let $C_i$ be the irreducible component of $\Gamma_i$ meeting with
$C_0$,  and let $n_i$  be the multiplicity of $C_i$ in $F$. Let
$\beta_{i}$ be the contribution  of $\Gamma_i$ to $\beta_F^{-}$. If
$r\geq 2$, then we have
\begin{align}\label{beta_i}
3=\sum\limits_{i=1}^{r}\frac{n_i}{n_0},\qquad
\beta_F^{-}=\sum\limits_{i=1}^{r}\beta_i<\frac{5}{2}.
\end{align}

\begin{lemma} {\rm 1)}    $r\leq 3$, and
if $r=3$, then $F$ is the fiber of type 18. 

 {\rm 2)} If $r=2$ and all $\Gamma_i$ are not of type $A_n$, then $F$ is of types 5, 19, 20 and 22.

{\rm 3)} If $r=2$ and $\Gamma_1$ is of type $A_n$, then $F$ is is of
types 1, 2, 4, 6, 9 and 17.

{\rm 4)} If $r=1$, then $F$ is of types 3, 7 and 8.
\end{lemma}
\begin{proof}
1) Note first that if all $\Gamma_i$ are of type $A_n$ and form some
H-J branches, then by a straightforward computation, we have
$\beta_F^{-}=\sum_{i=1}^{r}\frac{n_i}{n_0}=3>\frac{5}{2}$, a
contradiction.
 Since $\beta_F^{-}<\frac{5}{2}$, $r\leq 4$. If $r=4$, then all
$\Gamma_i$ are of type $A_n$ and form H-J branches, impossible.
 Hence $r\leq 3$.

 Assume $r=3$. $\beta_F^{-}<\frac{5}{2}$ implies $F_{\rm{red}}-C_0$ contains two
 H-J branches of type $A_n$, say $\Gamma_1$ and  $\Gamma_2$. By
 \eqref{beta_i}, we have
 \begin{align*}
 \beta_F^-&=\frac{n_1}{n_0}+\frac{n_2}{n_0}+\beta_3
          =3-\dfrac{n_3}{n_0}+\beta_3.
 \end{align*}
If $\Gamma_3$ is of types $E_7$, $E_6$ or $D_m^*$, then one can
check that $\beta_F^-=\frac83$, $\frac{17}6$ and $4$ respectively,
which contradicts the condition $\beta_F^-<\frac52.$ If $\Gamma_3$
is of type $A_{k+l+1}$ as above, then $k\geq 1$ and $l\geq 1$ (since
$\Gamma_3$ is not a  H-J branch of type $A_n$). On the other hand,
$\frac{n_1}{n_0}\geq \frac12$ and $\frac{n_2}{n_0}\geq \frac12$, so
$\frac{n_3}{n_0}\leq 2$.   $n_3=(l+1)m=(k+1)n$, so
$\frac{n_3}{n_0}=\frac{(k+1)(l+1)}{(k+1)+(l+1)}\leq 2$, $(l\leq k)$,
we obtain that $l=1$. Hence
$$
\beta_F^-=3+\dfrac{k}{k+1}+\dfrac{1}{2}-\dfrac{2(k+1)}{k+3}=\dfrac52+\dfrac{3k+1}{(k+1)(k+3)}>\dfrac52,
$$
a contradiction.

Finally, assume that $\Gamma_3$ is of type $D_{k+3}$ as above.
$\beta_3=\frac12+\frac{k}{k+1}$, so
$$
\beta_F^-=3+\frac12+\frac{k}{k+1}-\dfrac{k+3}{4}<\dfrac52,
$$
we get $k\geq 5$. On the other hand,  $\frac{n_3}{n_0}\leq 2$, i.e.,
$\frac{k+3}{4}\leq 2$ and $k\leq 5$. Hence $k=5$,
$2n_1=2n_2=n_0=2n$. Because $F$ can not be a multiple fiber,
$n=1$. This is just the fiber of type 18. 

2), 3) and 4) can be proved by similar calculations.
\end{proof}

\subsection{Applications}

The local canonical class inequality has some interesting applications. It has
been used to establish the canonical class inequality for non-semistable fibrations.
Now we give a new proof of the following
well-known result.
\begin{corollary} Let $f:X\to \mathbb P^1$ be a nontrivial fibration
of genus $g\geq 1$. Then $f$ admits at least $2$ singular fibers.
\end{corollary}
\begin{proof} If $f$ is smooth, then it is trivial. Now we assume
that $f$ admits only one singular fiber $F$. In this case, $f$ is
isotrivial.  So
 \begin{equation}\label{2.1}
  c_1^2(X)=-8(g-1)+c_1^2(F), \hskip0.3cm
c_2(X)=-4(g-1)+c_2(F).
 \end{equation}

If $g\geq 2$, we proved in \cite{TTZ05} that
$c_1^2(X)+8(g-1)=K_f^2\geq 4(g-1)$. By (\ref{2.1}), we have
$c_1^2(F)\geq 4g-4$, a contradiction.

If $g=1$, then $12\chi(\mathcal O_X)=c_2(X)=c_2(F)$. So $c_2(F)$ is
divided by $12$. We know that $c_2(F)=0$, and $F=nE$ for some smooth
elliptic curve $E$ and $n\geq 2$. Hence $\chi(\mathcal O_X)=0$. By
the formula for canonical class, we have $K_X\sim -{(n+1)} E$. Hence
$X$ is birationally ruled, $p_g(X)=0$ and $q(X)=1$. The Albanese map
$\alpha:X\to B$ is the ruling. Let $F'$ be a fiber of $\alpha$. Then
$2=-K_XF'=(n+1)EF'\geq n+1\geq 3$, a contradiction.

This proves that $f$ admits at least 2 singular fibers.
\end{proof}
\subsection{Chern numbers of the fibers in Theorem \ref{THMnew1.3}}
In order to prove 2) of Theorem \ref{THMnew1.2}, we need to compute
the Chern numbers for all 22 singular fibers in Theorem
\ref{THMnew1.3}.
\\

{\center
\begin{tabular}{|c|c|c|c|c|c|c|c|c|c|c|c|c|c|c|c|c|c|}\hline
 \phantom{\footnotesize$\dfrac1{1}$} $F$\phantom{\footnotesize$\dfrac1{1}$} & 1 & 2
  & 3 &4 & 5 & 6
   & 7 & 8 &9 &10& 11\\
   \hline
  \phantom{\footnotesize$\dfrac11$}$g$\phantom{\footnotesize$\dfrac11$} & 6 & 4
  & 3 &3 & 3 & 3
   & 2 & 2 &2 &2 &2\\
  \hline
  \phantom{$\dfrac11$}$c_1^2$\phantom{$\dfrac11$} &$\frac{130}{7}$
  & $\frac{54}{5}$ & 7 &$\frac{48}7$  &$\frac{98}{15}$  &$\frac{20}{3}$&$\frac{16}{5}$  & 3&3 &3 &$\frac{8}{3}$\\ \hline

  \phantom{$\dfrac11$}$c_2$\phantom{$\dfrac11$} & 30 & 26 & 21 & 18& $\frac{268}{15}$& 20  & 16  &15 &15 &9 &$\frac{34}{3}$\\ \hline

  \phantom{$\dfrac11$}$\chi$\phantom{$\dfrac11$} & $\frac{85}{21}$ &$\frac{46}{15}$  &$\frac73$  & $\frac{29}{14}$& $\frac{61}{30}$ &
 $\frac{20}{9}$ &$\frac{8}{5}$   &$\frac{3}{2}$
 &$\frac{3}{2}$ & 1&$\frac{7}{6}$\\ \hline
 \end{tabular}\\[0.3cm]}

 {\center
\begin{tabular}{|c|c|c|c|c|c|c|c|c|c|c|c|c|c|c|c|c|c|}\hline
  \phantom{\footnotesize$\dfrac11$}$F$\phantom{\footnotesize$\dfrac11$} & 12&13&14&15&16&17&18&19&20&21&22\\
   \hline
  \phantom{\footnotesize$\dfrac11$}$g$\phantom{\footnotesize$\dfrac11$} &2&2&2&2&2&2&2&2&2&2&2\\
  \hline
  \phantom{$\dfrac11$}$c_1^2$\phantom{$\dfrac11$} &$\frac{11}{4}$  & $\frac{17}{6}$ & $\frac{8}{3}$ & $\frac{11}{4}$ & $\frac{17}{6}$ &
 $\frac{14}{5}$ &$\frac{8}{3}$
  &$\frac{8}{3}$
 &$\frac{8}{3}$ &$\frac{13}{5}$ &$\frac{31}{12}$\\ \hline
  \phantom{$\dfrac11$}$c_2$\phantom{$\dfrac11$} &$\frac{49}{4}$  & $\frac{79}{6}$ &$\frac{34}{3}$  &$\frac{49}{4}$ &$\frac{79}{6}$  &
 14 & $\frac{40}{3}$  &$\frac{40}{3}$
 &$\frac{52}{3}$ &7 &$\frac{197}{12}$\\ \hline
 \phantom{$\dfrac11$}$\chi$\phantom{$\dfrac11$} & $\frac{5}{4}$ &$\frac{4}{3}$  &$\frac{7}{6}$  &$\frac{5}{4}$ & $\frac{4}{3}$ &
 $\frac{7}{5}$ & $\frac{4}{3}$  &$\frac{4}{3}$
 &$\frac{5}{3}$ &$\frac{4}{5}$ &$\frac{19}{12}$\\ \hline
 \end{tabular}\\[0.7cm]}

\section{Proof of Theorem \ref{THMnew1.4}}
In this section, we will classify all singular fibers satisfying $
2c_2(F)-c_1^2(F)<6$.
\begin{lemma}\label{lemma6.1} If $2c_2(F)-c_1^2(F)\neq 0$, then
$2c_2(F)-c_1^2(F)\geq 3$.
\end{lemma}
\begin{proof} We have
\begin{align}\label{eq2c2c12tan}
2c_2(F)-c_1^2(F)=2(\mu_F-\beta_{F}-\alpha_F)+\alpha_F+3\beta_F^{-}-F_{\rm{red}}^2<3.
\end{align}
In particular, we have $\sum_{p\in F}\left(
2(\mu_p-\beta_{p}-\alpha_p)+\alpha_p+3\beta_p^{-}\right)<3+F_{\rm
red}^2\leq 3$. By Lemma \ref{mualphabeta}, 4), $F$ is a nodal curve,
and so $\alpha_F=0$. If $F_{\rm red}^2=0$, then $F=nF_{\rm red}$ and
$2c_2(F)-c_1^2(F)=0$. If $F_{\rm red}^2\leq -1$, then
$\mu_F-\beta_F<1$. Note that if a node $p$ satisfies $\beta_p\neq
1$, then $\beta_p\leq \frac12$, and $\mu_p-\beta_p\geq \frac12$.
Hence $\mu_F-\beta_F<1$ implies that at most one node $p$ satisfies
$\beta_p\neq 1$. So $F=nA+mB$, $A$ and $B$ are reduced nodal curve
and $AB=1$. By Zariski's lemma, $0=AF=nA^2+mAB=nA^2+m$, similarly,
$n+mB^2=0$. Hence $m=n$, and $F_{\rm red}^2=0$, a contradiction.
This proves the lemma.
\end{proof}

\noindent {\it Proof of Theorem \ref{THMnew1.4}}: We can assume that
$2c_2(F)-c_1^2(F)\geq 3$
\begin{align}\label{eq2c2c12}
3\leq
2(\mu_F-\beta_{F}-\alpha_F)+\alpha_F+3\beta_F^{-}-F_{\rm{red}}^2<6.
\end{align}
In particular, we have $\sum_{p\in F}\left(
2(\mu_p-\beta_{p}-\alpha_p)+\alpha_p+3\beta_p^{-}\right)<6+F_{\rm
red}^2\leq 6$. By Lemma \ref{mualphabeta}, 4), $F$ admits at most
one singular point $p$ which is not a node, and $p$ is of types
$A_2$, $A_3$ or $D_4$.

 If $F^2_{\rm red}=0$, then $F=nF_{\rm red}$. Then one can compute
all the local invariants directly, and we get the cases 2) $\sim$
5). In what follows we assume that $F_{\rm red}^2\leq -1$. So
$\sum_{p\in F}\left(
2(\mu_p-\beta_{p}-\alpha_p)+\alpha_p+3\beta_p^{-}\right)<5$, and $p$
is at worst of $A_3$.

 Let $s$
 be the number of nodes in $F_{\rm red}$ satisfying $\beta_q<1$. For such a node
 $q$, the two components of $F$ at $q$ have distinct multiplicities. So
 $\beta_q\leq\frac12$ and $\mu_q-\beta_q\geq
 \frac{1}{2}$.

If the non-nodal singular point $p$ exists, then $p$ is of types
$A_2$ or $A_3$. $2(\mu_p-\beta_{p}-\alpha_p)+\alpha_p+3\beta_p^{-}$
is at least $\frac72$ if $p$ is of types $A_2$ or $A_3$, so we get
$\frac{7}{2}+s<5$ , i.e., $s\leq 1$. As in the proof of the previous
lemma, $s=1$ is impossible by Zariski's lemma. So $s=0$, i.e., the
multiplicities of the two local branches of any node are the same.
Hence $p$ is of type $A_3$.

From Zariski's lemma,
 one has a decomposition $F=n(A+2B)$,  where $A$ and $B$ are connected reduced nodal
 curves and smooth at $p$,
 $A\cap B=\{p\}$, $A^2=-4$, $B^2=-1$, $AB=2$. We
 get case 8).

From now on  we always assume $F_{\rm red}$ is a nodal curve. By
assumption, $s\neq 0$. As in the proof of the previous lemma, $s=1$
is also impossible. So $2\leq s< 6+F_{\rm red}^2\leq 5$
 by \eqref{eq2c2c12}, i.e., $2\leq s\leq 4$.

Let $F=\gamma_1\Gamma_1+\cdots+\gamma_r\Gamma_r$, $r\geq 2$, where
$\Gamma_i$'s are reduced with $\Gamma_i^2=-e_i\leq -1$ (not
necessarily irreducible) and have no pairwise common components.
$\gamma_1,\cdots,\gamma_r$ are pairwise distinct. Then we have
$$r-1\leq \sum_{i<j}\Gamma_i\Gamma_j=s, \hskip0.3cm F_{\rm red}^2=2s-e_1-\cdots-e_r.$$
If $r-1=s$, then $\Gamma_1$, $\cdots,$ $\Gamma_r$ form a chain.
Assume that this is a chain like the one before Lemma \ref{gamma},
where $\gamma_0=\gamma_{r+1}=0$. So the liner equation
\eqref{LinearEquation} holds true. Since $\gamma_0=0$, we can see
from the equation that $\gamma_1$ divides $\gamma_i$ for any $i$.
Symmetrically, from $\gamma_{r+1}=0$, we know that $\gamma_r$
divides $\gamma_i$ for all $i$. So $\gamma_1=\gamma_r$, which
contradict our assumption. So $r\leq s$.

 Suppose $s=2$. Then $r=2$.
 Now one can prove that $F=nA+2nB$, $AB=2$, $A^2=-4$ and $B^2=-1$.
 We get cases 6) and  7).

Suppose  $s=3$. Then $F_{\rm red}^2\geq -2$ and $r=2$ or $3$.

 If
$r=2$, then one can prove that
$F=\gamma_1\Gamma_1+3\gamma_1\Gamma_2$, $\Gamma_1\Gamma_2=3$. Hence
$\mu_F-\beta_F=3-1=2$. Now we have $2(\mu_F-\beta_F)\geq 6+F_{\rm
red}^2$, a contradiction.

 If $r=3$ and $\Gamma_1\Gamma_2=\Gamma_2\Gamma_3=\Gamma_3\Gamma_1=1$,
then one can prove that
$F=\gamma_1\Gamma_1+3\gamma_1\Gamma_2+2\gamma_1\Gamma_3$. Hence
$\mu_F-\beta_F=3-1=2$ and $2(\mu_F-\beta_F)\geq 6+F_{\rm red}^2$, a
contradiction.

 If $r=3$, $\Gamma_1\Gamma_2=2$, $\Gamma_2\Gamma_3=1$ and
 $\Gamma_3\Gamma_1=0$, then we have $-e_1\gamma_1+2\gamma_2=0$,
 $2\gamma_1-e_2\gamma_2+\gamma_3=0$ and $\gamma_2-e_3\gamma_3=0$.
 Since $\gamma_2\neq \gamma_3$ , we have $e_3\geq 2$. We obtain that
 $e_1=e_3(e_1e_2-4)$, it implies $e_2=1$. $e_1=e_3=5$, or $e_1=6$
 and  $e_3=3$,  or $e_1=8$ and $e_3=2$. Hence $F_{\rm red}^2\leq
 -4$, a contradiction.

Suppose $s=4$. Then $F_{\rm red}^2=-1$. Let $F=\sum_{i=1}^k n_iC_i$,
where $C_i$ is irreducible.
 Let $q_1,\cdots, q_4$ be the nodes satisfying $\frac12\geq \beta_{q_1}\geq\cdots\geq \beta_{q_4}$.
By \eqref{eq2c2c12}, we have
\begin{align}\label{0.4}
2\sum_{=1}^4  (\mu_{q_k}-\beta_{q_k})<5,
\end{align}
It implies $\beta_{q_1}=\beta_{q_2}=\frac{1}{2}$.
 Let $C_{i_1}$ and $C_{i_2}$ be the components passing
 through $q_i$ ($i\leq 4$) and let $n_{i1}\leq n_{i2}$.  From $\beta_{q_1}=\beta_{q_2}=\frac{1}{2}$,
 we have   ${n_{i2}}=2n_{i1}$ for $i=1$ and $2$, and thus
 $\chi(n_{11}, n_{12})=\chi(n_{21}, n_{22})=0$.

From the proof of Theorem \ref{theorem_chi_F}, we have
$$ \sum_{k\leq 4} (\mu_{q_k}-\beta_{q_k})+F_{\rm red}^2=-12\sum_{i\leq 4}\chi(n_{i1}, n_{i2}).$$
Since $F_{\rm red}^2=-1$, and $\beta_{q_1}=\beta_{q_2}=\frac12$,
$$\mu_{q_3}-\beta_{q_3}+\mu_{q_4}-\beta_{q_4}=
6-\left(\frac{n_{31}}{n_{32}}+\frac{n_{32}}{n_{31}}\right)-
\left(\frac{n_{41}}{n_{42}}+\frac{n_{42}}{n_{41}}\right)-\beta_{q_3}-\beta_{q_4}.$$
i.e.,
$$4=\left(\frac{n_{31}}{n_{32}}+\frac{n_{32}}{n_{31}}\right)+
\left(\frac{n_{41}}{n_{42}}+\frac{n_{42}}{n_{41}}\right)\geq
2+2=4.$$ Thus $n_{31}=n_{32}$ and $n_{41}=n_{42}$, so
$\beta_{q_3}=\beta_{q_4}=1$, a contradiction.

Up to now  we have completed the proof. \hfill $\Box$

\begin{corollary}\label{last}
If the semistable model of $F$ is smooth and $F$ is not the multiple
of a smooth curve, then $2c_2(F)-c_1^2(F)\ge 6$.
\end{corollary}

\begin{corollary} {\rm\cite{Po08}}
If $f: X\to C$ is an isotrivial family of curves, then $K_X^2\ne
8\chi(\mathcal{O}_X)-1$.
\end{corollary}
\begin{proof} In this case, the modular invariants of $f$ are zero.
 Hence
$$
2c_2(X)-c_1^2(X)=\sum_i\left(2c_2(F_i)-c_1^2(F_i)\right)\geq0.
$$
Suppose $K_X^2\neq 8\chi(\mathcal{O}_X)$, i.e., $2c_2(X)\neq
c_1^2(X)$, then at least one singular fiber satisfies the condition
of Corollary \ref{last}, hence $2c_2(X)-c_1^2(X)\geq 6$,
equivalently, $K_X^2\leq 8\chi(\mathcal{O}_X)-2$.
\end{proof}

\begin{corollary}\label{propc2c2c2}
If $F$ is not semistable,
 then   $c_2(F)\ge \frac{11}{6}$ and $\chi_F\geq\frac16$.
One of the equalities holds if and only if $F$ is a reduced curve
with one ordinary cusp and some nodes.
 \end{corollary}
\begin{proof}
If $2c_2(F)-c_1^2(F)\geq 6$, equivalently, $8\chi_F-c_1^2(F)\geq 2$,
then $c_2(F)>3$ and $\chi_F > \frac14$. So we can assume that
$2c_2(F)-c_1^2(F)<6$. By Theorem \ref{THMnew1.4}, we have 8 types of
singular fibers. We see that only type 2) fiber with $N_F=0$ has the
minimal $c_2(F)$ and $\chi_F$. This proves the corollary.
\end{proof}

\bigskip \noindent {\it Questions}: 1) What is the upper bound of
$c_2(F)$. We conjecture $c_2(F)\leq \frac{55g}{6}$.

2) Is $\frac16$ the lower bound of $c_1^2(F)$ for a minimal
non-semistable fiber $F$?

3)  Is the inequality $c_1^2(F)\geq \chi_F$ true for any minimal
singular fiber $F$? \hskip0.2cm (If $F$ is a singular fiber in an
isotrivial family, then one can prove easily that $c_1^2(F)\geq
\frac{4(g-1)}{g}\chi_F$).


\begin{thebibliography}{[g]}

\bibitem{Artin} M. Artin:  {\it On isolated rational singularities of surfaces}, Amer. J. Math., {\bf 88} (1966),
 129--136.

\bibitem{AK00} T. Ashikaga and K. Konno:  {\it Global and local
properties of pencils of algebraic curves},  Algebraic Geometry
2000, Azumino, Advanced Studies in Pure Mathematics  {\bf 36}
(2000),
 1--49.

\bibitem{Be} A. Beauville:  {\it L'in\'egalit\'e  pour les surfaces de type g\'en\'erale},
Appendix to: O. Debarre, In\'egalit\'es num\'eriques pour les
surfaces de type general, Bull. Soc. Math. de France, {\bf 110}
(1982), 319-346.

\bibitem{BPV} W. Barth, C. Peter, A. Van de ven: {\itshape Compact complex
surfaces}, Berlin, Heidelberg, New York: Springer, 1984.

\bibitem{Iitaka} S. Iitaka:
 {\it Master degree thesis},
   University of Tokoyo
  {\bf } (1967).

\bibitem{Kod1963III} K. Kodaira:
 {\it On compact analytic surfaces}, III,
   Ann. of Math.
  {\bf 78} (1963), no.~1, 1--40.

\bibitem{NaUe} Y. Namikawa, K. Ueno:
 {\it On fibers in families of curves of genus two}, I,
   Algebraic Geometry and Commutative Algebra, in honor of Y.
   Akizuki, Kinokuniya, Tokoyo
  {\bf } (1973), 297--371.

\bibitem{Ogg} A. P. Ogg:
 {\it On pencils of curves of genus two},
  Topology,
  {\bf 5} (1966), 355--362.

\bibitem{Po08} F. Polizzi: {\itshape Numerical properties of isotrivial
 fibrations}.  arXiv:0810.4195.

 \bibitem{Ta94} S.-L. Tan: {\it On the base changes of penciles of curves}, I,
 Manus. Math.,  {\bf 84} (1994), 225--244.

\bibitem{Ta95} S.-L. Tan: {\itshape The minimal number of singular fibers of a
semistable curve over $\mathbb{P}^1$}, J. Algebraic Geom. {\bf 4}
(1995), 591-596.

 \bibitem{Ta96} S.-L. Tan: {\itshape On the base changes of penciles of
curves}, II, Math. Z., {\bf 222} (1996), 655--676.

\bibitem{Ta98} S.-L. Tan: {\itshape On the slopes of the moduli
spaces of curves}, Intern. J. of Math., {\bf 9} (1998), 119--127.


 \bibitem{Ta10} S.-L. Tan: {\itshape Chern numbers of a singular fiber, modular invariants
and isotrivial  families of curves}, Acta Math. Viet., (to appear).

\bibitem{TTZ05} S.-L. Tan, Y.-P. Tu, and A.-G. Zamora: {\itshape On complex
surfaces with $5$ or $ 6$ semistable singular fibers over
$\mathbb{P}^1$}, Math. Zeit. {\bf 249} (2005), 427--438.

\bibitem{Uematsu} K. Uematsu:
 {\it Numerical classification of singular fibers in genus $3$ pencils},
   J. Math. Kyoto Univ.
  {\bf 39-4} (1999), 763--782.

\bibitem{Xi90} G. Xiao: {\it On the stable reduction of pencils of
curves}, Math. Z., {\bf 203}  (1990), 379-389.

\bibitem{Xi92}G. Xiao: {\itshape The fibrations of algbraic surfaces},
Shanghai Scientific \&  Technical Publishers, 1992 (in Chinese).














\vskip0.5cm

{
Department of Mathematics, East China Normal University, \\
 Dongchuan RD 500,
 Shanghai 200241, P. R. of China \\
jlu@math.ecnu.edu.cn \hskip0.5cm
 sltan@math.ecnu.edu.cn
 }

















\end{thebibliography}
\end{document}